\documentclass[12pt]{amsart}
\usepackage{amsmath,amsthm,amscd,amssymb,amsfonts, amsbsy}
\usepackage{latexsym, color, enumerate}
\usepackage{esint}
\usepackage{mathrsfs}
\usepackage{marginnote}
\usepackage{todonotes}
\usepackage{hyperref}

\usepackage[top=1in, left=1in, right=1in, bottom=1in]{geometry}

\usepackage{tikz}
\usepackage{multicol}

\usepackage{cancel}

\numberwithin{equation}{section}

\theoremstyle{plain}
\newtheorem{theorem}{Theorem}[section]
\newtheorem{lemma}[theorem]{Lemma}

\newtheorem{proposition}[theorem]{Proposition}

\theoremstyle{definition}
\newtheorem{definition}[theorem]{Definition}

\newtheorem{assumption}[theorem]{Assumption}

\theoremstyle{remark}
\newtheorem{remark}[theorem]{Remark}

\newcommand{\bN}{\mathbb{N}}

\newcommand{\bR}{\mathbb{R}}

\renewcommand{\vec}[1]{\boldsymbol{#1}}

\def\XXint#1#2#3{{\setbox0=\hbox{$#1{#2#3}{\int}$}
		\vcenter{\hbox{$#2#3$}}\kern-.5\wd0}}

\newcommand{\p}{\partial}
\newcommand{\epsi}{\varepsilon}
\newcommand{\dist}{\operatorname{dist}}

\begin{document}

\title[]{Asymptotic expansions for harmonic functions at conical boundary points}

\author[D. Kriventsov]{Dennis Kriventsov}
\address[D. Kriventsov]{Department of Mathematics, Rutgers University, 110 Frelinghuysen Road, Piscataway, NJ 08854-8019, USA}
\email{dnk34@math.rutgers.edu}

\author[Z. Li]{Zongyuan Li}
\address[Z. Li]{Department of Mathematics, Rutgers University, 110 Frelinghuysen Road, Piscataway, NJ 08854-8019, USA}
\email{zongyuan.li@rutgers.edu}


\begin{abstract} We prove three theorems about the asymptotic behavior of solutions $u$ to the homogeneous Dirichlet problem for the Laplace equation at boundary points with tangent cones. First, under very mild hypotheses, we show that the doubling index of $u$ either has a unique finite limit, or goes to infinity; in other words, there is a well-defined order of vanishing. Second, under more quantitative hypotheses, we prove that if the order of vanishing of $u$ is finite at a boundary point $0$, then locally $u(x) = |x|^m \psi(x/|x|) + o(|x|^m)$, where $|x|^m \psi(x/|x|)$ is a homogeneous harmonic function on the tangent cone. Finally, we construct a convex domain in three dimensions where such an expansion fails at a boundary point, showing that some quantitative hypotheses are necessary in general. The assumptions in all of the results only involve regularity at a single point, and in particular are much weaker than what is necessary for unique continuation, monotonicity of Almgren's frequency, Carleman estimates, or other related techniques.
\end{abstract}

\maketitle

\section{Introduction}

Let $\Omega\subset \bR^d$ be an open set with $0\in\p\Omega$, and consider solutions to the Dirichlet problem for the Laplace equation:
\begin{equation}\label{eqn-230314-1018}
	\begin{cases}
		\Delta u = 0 \quad &\text{in}\,\,\Omega\cap B_2,\\
		u = 0 \quad &\text{on}\,\,\p\Omega\cap B_2.
	\end{cases}
\end{equation}
A basic question here is to understand the asymptotic behavior of $u$ near $0$. If instead $0 \in \Omega$ was an interior point, the asymptotic behavior is clear---as $u$ is analytic, locally it can be decomposed as a leading-order homogeneous harmonic function plus higher order terms. Similar expansion formulas hold in related contexts where analyticity is not available, including second-order or higher order elliptic equations with $C^\infty,$ Lipschitz, or H\"older coefficients; see \cite{MR75416, MR1305956, MR1466314}. 

Below, we will use the term \emph{asymptotic expansion} loosely for representations $u = v + w$, where $v$ has homogeneity $m$ (or more generally growth at least $g(|x|)$ for some nondecreasing $g\geq 0$) and $w = o(|x|^m)$ (or more generally $w = o(g(|x|))$ for the same modulus $g$), possibly with more quantitative control over $w$.

At boundary points (so $0 \in \p \Omega$), Felli and Ferrero proved in \cite{MR3109767} that if $\Omega$ is a $C^{1,\alpha}$ perturbation of a regular cone, then suitable rescalings of the solution (i.e. blow-ups) converge to a nontrivial homogeneous harmonic function on that cone; this gives an expansion of $u$ to leading order.
In \cite{MR4525136}, Kenig and Zhao show that when $\p\Omega \in C^{1,\text{Dini}}$, the expansion formula
$u(x) = P_N(x) + O(|x|^{N} \int_0^{|x|}\widetilde{\omega})$
holds, where $\widetilde{\omega}$ is a modified modulus of continuity for the unit normal and $P_N$ is a degree $N$ homogeneous harmonic polynomial on $\bR^d_+$ that vanishes on the boundary (choosing coordinates so that $\partial \mathbb{R}^d_+$ is tangent to $\p \Omega$ at $0$).

In these results, $\Omega$ is required to have certain smoothness in a neighborhood of $0$. Here, we aim to discuss rougher domains under only one-point conditions, i.e. conditions which do not imply any smoothness except possibly at the single point $0$ itself. It is clear that at least some assumptions are necessary at $0$ to have any hope for solutions to \eqref{eqn-230314-1018} to have asymptotic expansions. Indeed, assuming $\Omega$ is regular for the Dirichlet problem, consider a Green's function $G$ for $\Omega$ with some fixed pole: if $G(x) = |x|^m \psi(x/|x|) + o(|x|^m)$, then $\Omega = \{G(x) > 0\}$ is tangent to the cone $\Gamma = \{x : \psi(x/|x|) > 0\}$ at the origin. With this in mind, we begin with the following definitions.

\begin{definition} \label{def-230315-0126}
	Given an open cone $\Gamma$, the point $0 \in \p \Omega$ is called \emph{conical with cone} $\Gamma$ if, in Hausdorff distance,
	\[
		r^{-1}\dist(\p\Omega\cap B_r, \p \Gamma \cap B_r) \rightarrow 0.
	\]
	Here, a \emph{cone} $\Gamma$ is a set invariant under dilation, i.e. $r \Gamma = \Gamma$ for all $r > 0$. 
\end{definition}

For an open cone $\Gamma$, let
\begin{equation*}
	(0<) \lambda_{1,\Gamma} \leq \lambda_{2,\Gamma} \leq \cdots \lambda_{k,\Gamma} \leq \cdots
\end{equation*}
be the sequence of Dirichlet eigenvalues (counting multiplicity) of the Laplace-Beltrami operator on the cross-section $\Gamma \cap B_1$. Also let $\psi_{k,\Gamma}$ be the associated eigenfunction, and 
\begin{equation}\label{eqn-221018-0432}
	m_{k,\Gamma} = \frac{-(d-2) + \sqrt{(d-2)^2 + 4 \lambda_{k,\Gamma}}}{2} (>0)
\end{equation}
be the characteristic constant. Then all homogeneous harmonic function on $\Gamma$ vanishing on $\p \Gamma$ can be written as $|x|^{m_{k,\Gamma}} \psi_{\lambda_{k,\Gamma}}(x/|x|)$. Below, the dependence on $\Gamma$ will be omitted when there is no ambiguity. It is worth mentioning that in Definition \ref{def-230315-0126}, $\Gamma$ is allowed to be $\bR^d_+$ so in particular every boundary point of a $C^1$ domain is conical.

If one wants to find an asymptotic expansion for $u$ at $0$, the first step is to identify the homogeneity of the leading-order term. On cones, convex domains, or sufficiently regular perturbations of them, the Almgren frequency gives a way to read off this homogeneity (this will be discussed below), but in this more general configuration it is unclear that the frequency is even approximately monotone. Instead, taking zero extensions of solutions to \eqref{eqn-230314-1018} outside $\Omega$, define the doubling index
\begin{equation}\label{eqn-doublingindex}
	N_u(r) := \frac{\fint_{\p B_r} |u|^2}{\fint_{\p B_{r/4}} |u|^2}.
\end{equation}
We say that $u$ is \emph{asymptotically homogeneous} if the limit of $N_u(r)$ at $r = 0$ exists in the extended real number sense ($\lim_{r \searrow 0} N_u(r) \in [0, \infty]$). The doubling index is a rough measure of homogeneity, and the existence of this limit means that there is a unique leading-order homogeneity for $u$ in this rough sense. Our first theorem states that $u$ is asymptotically homogeneous if $0$ is a conical point, under a smoothness assumption on $\Gamma$:
\begin{assumption} \label{ass-230707-1023}
	The open cone $\Gamma$ is graphical, in the sense that $\Gamma = \{(x', x_n) : x_n > g(x') \}$ for some choice of coordinates and function $g : \mathbb{R}^{d-1}\rightarrow \mathbb{R}$, and $g$ is Lipschitz.
\end{assumption}

\begin{theorem} \label{thm-230225-1127}
	If $0\in\p\Omega$ is conical with cone $\Gamma$ satisfying Assumption \ref{ass-230707-1023}, then $u$ is asymptotically homogeneous and, moreover,
	\begin{equation*}
		\lim_{r\searrow 0} N_u(r) = 4^{2m} \,\, \text{for some} \,\, m \in \{m_{\lambda_{k,\Gamma}}\}_{k=1}^\infty\cup \{ +\infty \},
	\end{equation*}
where $m_{\lambda_{k,\Gamma}}$ are the characteristic constants of $\Gamma$ defined in \eqref{eqn-221018-0432}.
\end{theorem}
Note that Theorem \ref{thm-230225-1127} does not exclude the possibility that $u$ vanishes to infinite order near $0$, i.e., strong unique continuation property (SUCP) fails. Assumption \ref{ass-230707-1023} can be considerably relaxed, for example to a uniform Lebesgue density condition on $\Gamma^c$, but we do not attempt maximal generality here: in fact, this theorem is interesting even when $\Gamma = \bR^d_+$.

Theorem \ref{thm-230225-1127} suggests that one might consider Almgren blow-ups of $u$, the rescaled functions
\begin{equation} \label{eqn-230627-0304}
	u_r(y) = \frac{u(r y)}{(\fint_{\p B_r \cap \Omega} |u|^2)^{1/2}},
\end{equation}
to attempt to find the leading order term in an asymptotic expansion for $u$, even when the Almgren frequency is unavailable. Indeed, the boundedness of $N_u(r)$ is enough to guarantee the compactness of $\{u_r\}_{r\in (0,1)}$ in $L^2$. Hence, along subsequences $r_k \rightarrow 0$, $u_r$ converges. Moreover, $\lim_{r\searrow 0}N_u(r) = 4^{2m} < \infty$ guarantees that the blow-up limit has to be a homogeneous harmonic function with homogeneity $m$. When $d=2$, one can further obtain the uniqueness of the blow-up limit simply due to the fact that the eigenvalues $\lambda_{k, \Gamma}$ are all simple, giving an asymptotic expansion. This approach is discussed Section \ref{sec-221223-0602}. 

When $d\geq 3$, however, it turns out we require additional assumptions. First, the following essentially says that $\p \Omega$ is $C^{1, \alpha}$ at $0$ only, but with an arbitrary tangent object:
\begin{definition}[$\alpha$-conical]
	Given an open cone $\Gamma$ we say $0$ is \emph{$\alpha$-conical with cone $\Gamma$} if there exists $\alpha > 0$ such that
	\begin{equation}\label{eqn-220728-0546}
		\limsup_{r\rightarrow 0} r^{-(1+\alpha)} \dist(\p\Omega\cap B_r, \p\Gamma\cap B_r) <\infty.
	\end{equation}
\end{definition}
We will also need to assume some smoothness of the limit cone $\Gamma$.
\begin{assumption} \label{ass-230119-0339}
	The open cone $\Gamma$ is graphical in the sense of Assumption \ref{ass-230707-1023} with the graph $g$ being either $C^{1, \text{Dini}}$ or semiconvex.
\end{assumption}
Recall that a function $g$ is called semiconvex, if there exists a constant $C>0$, such that locally $g(x + y) - 2g(x) + g(x-y) \geq - C |y|^2$.
\begin{theorem} \label{thm-221218-1120} 
	Let $\Omega \subset \bR^d$ with $0\in\p\Omega$ and $d\geq 3$. If $0$ is $\alpha$-conical with cone $\Gamma$ satisfying Assumption \ref{ass-230119-0339}, then for any nontrivial solution $u$ to \eqref{eqn-230314-1018}, either $\lim_{r\rightarrow 0} N_r(u) = +\infty$ or there exists some $N \in \mathbb{N}$, $C \neq 0$, and $\alpha_0 \in (0,1)$ such that
	\begin{equation} \label{eqn-221228-0956}
		u(x) = C|x|^{m_{N,\Gamma}} \psi_{\lambda_{N,\Gamma}}(x / |x|) + w(x), \quad \text{where} \,\, |w(x)| \leq C |x|^{ m_{N, \Gamma} + \alpha_0 \alpha }.
	\end{equation}
\end{theorem}

It is easy to see that \eqref{eqn-221228-0956} implies $(\fint_{\p B_r \cap \Omega} |u|^2)^{1/2} \approx r^{m_{N,\Gamma}}$ and $u_r \rightarrow C |x|^{m_{N,\Gamma}} \psi_{N,\Gamma}$ in $L^2$.

\begin{remark}
	For $\p\Omega \in C^{1,Dini}$ or convex, it is known that (SUCP) holds, any non-trivial solution must vanish to at most finite order, and in particular $N_r(u)$ is bounded. On the other hand, the assumptions in Theorem \ref{thm-221218-1120} are weaker than any known criterion for (SUCP) even if the cone $\Gamma$ is a half-space, as far as we are aware.
\end{remark}

\begin{remark} 
	We expect similar results in Theorem \ref{thm-221218-1120} hold for operators with scaling subcritical coefficients and lower order terms, i.e.,
	\begin{equation*} 
		Lu = D_i(a_{ij}D_j u + \widetilde{W}_i u) + W_i D_i u + Vu
	\end{equation*}
	with $a_{ij} \in C^\epsi$, $\widetilde{W}_i, W_i \in L^{d+\epsi}_{loc}$ and $V \in L^{d/2 + \epsi}_{loc}$. See \cite{MR1466314} for an interior version which works for higher order elliptic equations with subcritical lower order terms.
\end{remark}

One may naturally ask whether the extra convergence rate condition in \eqref{eqn-220728-0546} is necessary. We construct a \emph{convex} domain $\Omega \subset \mathbb{R}^3$ for which $0 \in \p\Omega$ is conical with cone $\mathbb{R}^{3}_+$, but for which no expansion \eqref{eqn-221228-0956} exists for some $u$:
\begin{theorem} \label{thm-230315-0259}
	There exists a convex domain $\Omega \subset \bR^3$ with $0\in\p\Omega$ being conical with tangent cone $\bR^3_+ = \{(x,y,z): z>0\}$,
	a solution $u$ to \eqref{eqn-230314-1018}, and a sequence $r_k \rightarrow 0$, such that
	\begin{equation*}
		\frac{ u( r_{2k+1} \cdot ) }{ (\fint_{\p B_{r_{2k+1}}} |u|^2)^{1/2} } 
		\rightarrow
		4\sqrt{2/\pi} xz, \quad \frac{ u(r_{2k}\cdot) }{ (\fint_{\p B_{r_{2k}}} |u|^2)^{1/2} } 
		\rightarrow 
		4\sqrt{2/\pi} yz.
	\end{equation*}
\end{theorem}
The point here is that $r^{-1} (\Omega \cap B_r) \rightarrow B_1^+$ slowly, with cross-sections $\p B_r \cap \Omega$ resembling ellipses with oscillating eccentricity. Then a suitably chosen $u$ can be made to have traces $u|_{\p B_r}$ ``rotate" between two second eigenfunctions of the Lapacian on $\p B_r \cap \mathbb{R}^3_+$.

Note that when $\Omega$ is convex, every point $x_0 \in \p\Omega$ is conical. Indeed, $(\Omega - x_0)/r$ always converges \emph{monotonically} to a cone $\Gamma_{x_0}$. Moreover, Almgren's frequency is monotone on convex domains (\cite{MR1363203}), so
	\begin{equation*} 
		N_{u}(r) \searrow N_u(0) = 4^{2m_{N, \Gamma_{x_0}}} \,\, \text{for some} \,\, N \in \bN, \,\, \text{as} \,\, r\searrow 0.
	\end{equation*}
So the conclusion of Theorem \ref{thm-230225-1127} for convex $\Omega$ holds, and in a stronger form: $\lim_{r \searrow 0} N_u(r) < \infty$ and (SUCP) is valid. Therefore the example of Theorem \ref{thm-230315-0259} shows that to have an asymptotic expansion in the weakest possible sense (uniqueness of limits for the Almgren rescalngs $u_r$), it is not sufficient to have monotonicity of the frequency, or (SUCP), or even monotonicity in the convergence of $\Omega/r$ to its tangent cone; some sufficiently summable rate of convergence appears to be needed. 

Similarly, it follows that the Dini condition in \cite{MR4525136} cannot be replaced by even very strong geometric assumptions like convexity. In the recent work \cite{KZ2}, counterexamples are constructed of barely non-$C^{1, \text{Dini}}$ domains admitting solutions to the Dirichlet problem with large singular sets, but in those examples $u$ still has unique Almgren blow-ups.

It is worth emphasizing that in both Theorems \ref{thm-221218-1120} and \ref{thm-230225-1127}, we only assume one-point conditions at $0$. Compared to earlier results in \cite{MR3109767, MR4525136}, we do not need any smoothness condition on $\p\Omega$ or its normal direction $\vec{n}$ in a neighborhood. We hope the methodology here could be useful when discussing asymptotic and unique continuation properties of harmonic functions on rough domains.

The paper is organized as follows. In Section \ref{sec-230703-0345}, we prove Theorem \ref{thm-230225-1127}. After collecting some preliminary facts about Green's functions on cones in Section \ref{sec-230703-0347}, we provide the proof of Theorem \ref{thm-221218-1120} in Section \ref{sec-230703-0348}. Finally in Section \ref{sec-221223-0602}, we discuss the uniqueness of Almgren blow-ups on $\Omega \subset \bR^2$, and construct the example in Theorem \ref{thm-230315-0259}.

\section{Asymptotic homogeneity at a conical point}\label{sec-230703-0345}

In this section, we prove Theorem \ref{thm-230225-1127}. The key idea is to combine a compactness argument motivated by \cite{MR3952693, 2022arXiv220313393L} and a rigidity result. Besides the usual doubling index $N_u(r)$ defined in \eqref{eqn-doublingindex}, the following version using averages over full balls rather than spheres will also be useful:
\begin{equation*}
	\widetilde{N}_u(r) := \frac{\fint_{B_r} |u|^2}{\fint_{B_{r/2}} |u|^2}.
\end{equation*}
If there is no ambiguity, we suppress the subscript: $N = N_u$, $\widetilde{N} = \widetilde{N}_u$. It is worth noting that if $u$ is harmonic, $u^2$ is subharmonic, and so from the mean value property both $N, \widetilde{N} \geq 1$. The following lemma shows that $N$ and $\widetilde{N}$ are comparable at adjacent scales.
	\begin{lemma} \label{lem-230624-1045}
		Let $v$ be a subharmonic function on $B_1 \subset \bR^d$, $d \geq 2$, with $v \not \equiv 0$ on any neighborhood of $0$. Then for some $C = C(d)$
		\begin{equation*}
			\widetilde{N}_v(r) \leq C N_v (r), \quad N_v (s) \leq C \Pi_{j=0}^3 \widetilde{N}_v(2^{1-j}r),\quad \forall r \in (0,1/2) \,\, \text{and} \,\, s\in (r/2,r).
		\end{equation*}
	\end{lemma}
\begin{proof}
	For the first inequality, by definition
	\begin{equation} \label{eqn-230624-1039}
		\widetilde{N}_v(r) 
		= 
		\frac{\fint_{B_{r}} |v|^2}{\fint_{B_{r/2}} |v|^2} 
		\leq
		C \frac{\fint_{B_{r}} |v|^2}{\fint_{B_{r/2}\setminus B_{r/4}} |v|^2}.
	\end{equation}
	By the mean value property of $v^2$, which is subharmonic, we obtain
	\begin{equation*}
		\text{RHS of \eqref{eqn-230624-1039}}
		\leq
		C \frac{\fint_{\p B_{r}} |u|^2}{\fint_{\p B_{r/4}} |u|^2}
		=
		C N_u (r).
	\end{equation*}
	The second inequality can be proved similarly:
	\begin{equation*}
		N_v (s)
		=
		\frac{\fint_{\p B_{s}} |u|^2}{\fint_{\p B_{s/4}} |u|^2} 
		\leq
		C \frac{\fint_{B_{2r} \setminus B_r} |u|^2}{\fint_{B_{r/8}} |u|^2}
		\leq
		C \frac{\fint_{B_{2r}} |u|^2}{\fint_{B_{r/8}} |u|^2}
		=
		C \Pi_{j=0}^3 \widetilde{N}_v(2^{1-j}r).
	\end{equation*}
\end{proof}

The rigidity result is given as follows.
\begin{lemma} \label{lem-230225-1124}
	Let $\Gamma \subset \bR^d$ be a cone with vertex at $0$ satisfying Assumption \ref{ass-230707-1023}, $d \geq 2$, and $v$ be a non-trivial solution to 
	\begin{equation*} 
		\begin{cases}
			\Delta v = 0 \quad &\text{in}\,\,\Gamma\cap B_2,\\
			v = 0 \quad &\text{on}\,\,\p\Gamma\cap B_2.
		\end{cases}
	\end{equation*}
Then both $N_v$ and $\widetilde{N}_v$ are non-decreasing for $r \in (0,2)$. Moreover, if either $N_v(t) = N_v(s)$ or $\widetilde{N}_v(t) = \widetilde{N}_v(s)$ for some $t>s$, then $u$ is homogeneous of degree $m_{j,\Gamma}$ with a characteristic constant defined in \eqref{eqn-221018-0432}. In particular, $N_v \equiv 16^{m_{j,\Gamma}}$, $\widetilde{N}_v \equiv 4^{m_{j,\Gamma}}$.
\end{lemma}
The proof of Lemma \ref{lem-230225-1124} is standard, by computing the derivatives of the (generalized) Almgren's frequency functions. See Appendix \ref{sec-230225-1124}. The rest of the section is devoted to the proof of Theorem \ref{thm-230225-1127}.
From now on, let $\Gamma$ be the tangent cone of $\Omega$ at $0$ and $m_j$ be the characteristic constant defined in \eqref{eqn-221018-0432}.
\subsection{Step 1}

We prove that $\liminf N_u(r) < \infty$ implies (SUCP). More precisely, we show
\begin{equation} \label{eqn-230226-1203}
	\limsup N_u(r) \leq C (\liminf N_u(r))^4.
\end{equation}
For this, we first prove the following \textbf{claim}:
for any number $\mu \notin \{m_j\}_j$, there exists $r_0 = r_0 (d,\mu,\Omega)$, such that
$\widetilde{N}_u(r) \leq  2^{2\mu}$ implies $\widetilde{N}_u(r/2) \leq  2^{2\mu}$ for all $r \in (0, r_0)$.
\begin{proof}[Proof of the claim]
	We argue by contradiction. Suppose to the contrary that there exist solutions $u_k \in H^1$ to \eqref{eqn-230314-1018} and $r_k \rightarrow 0$, such that
	$\widetilde{N}_{u_k}(r_k) \leq  2^{2\mu}$, $\widetilde{N}_{u_k}(r_k/2) >  2^{2\mu}$.
	Let
	\begin{equation*}
		\widetilde{u}_k(y) := u_k(r_k y) \big/ \left(\fint_{B_{r_k}} |u_k|^2 \right)^{1/2}.
	\end{equation*}
	Then we have 
	\begin{equation*}
		\Delta \widetilde{u}_k = 0 \,\, \text{in}\,\,r_k^{-1}(\Omega \cap B_{r_k}),
		\quad
		\widetilde{u}_k = 0 \,\, \text{on} \,\, B_1 \setminus r_k^{-1} (\Omega \cap B_{r_k}),
		\quad
		\text{and}\,\,
		\fint_{B_1} |\widetilde{u}_k|^2 = 1.
	\end{equation*}
	By the Caccioppoli inequality and Sobolev embeddings, for all $\epsi \in (0,1)$,
	\begin{equation} \label{eqn-230226-0916}
		\widetilde{u}_k \rightarrow u_\infty, \,\,\text{weakly in}\,\, L^2 (B_1), H^1(B_{1-\epsi}), \,\,\text{and strongly in}\,\, L^2(B_{1-\epsi})
	\end{equation}
	passing to a subsequence. Since $\Omega$ is conical at $0$, we claim that 
	\begin{equation} \label{eqn230226-0924}
			\Delta u_\infty = 0 \,\, \text{in}\,\,B_1 \cap \Gamma,
			\quad
			u_\infty = 0 \,\, \text{on}\,\, \p\Gamma \cap B_1.
	\end{equation}
	
	To see that $u_\infty$ is harmonic, take any test function $\varphi \in C^\infty_c (B_1\cap\Gamma)$. From Definition \ref{def-230315-0126}, for sufficiently large $k$, we have $\operatorname{supp}(\varphi) \subset r^{-1} (\Omega \cap B_{r_k})$. Combining with \eqref{eqn-230226-0916}, we obtain $\int \nabla \widetilde{u}_k \cdot \nabla \varphi = 0$. Passing $k\rightarrow \infty$ and noting that $\varphi$ is chosen arbitrarily, we obtain $\Delta u_\infty = 0$ in $B_1 \cap \Gamma$. For the boundary condition, we first note that $u_\infty = 0$ a.e. on $B_1 \setminus \bar{\Gamma}$ by recalling the $L^2(B_{1-\epsi})$ strong convergence. The desired zero boundary value now follows from the fact that $\Gamma$ is Lipschitz.
	
	Note that from \eqref{eqn-230226-0916}, we have
	\begin{equation*}
		\fint_{B_1} |u_\infty|^2 \leq \liminf_k \fint_{B_1} |\widetilde{u}_k|^2 = 1 \quad \text{and} \quad  \fint_{B_r} |u_\infty|^2 = \lim_k \fint_{B_r} |\widetilde{u}_k|^2,\,\, \forall r<1,
	\end{equation*}
	In particular, $\fint_{B_{ 1/2 }} |u_\infty|^2 \geq 2^{2\mu} > 0$, so $u_\infty \not\equiv 0$ on $\Gamma$. Moreover,
	\begin{equation*}
		\widetilde{N}_{u_\infty} (1) \leq \liminf_k \widetilde{N}_{\widetilde{u}_k} (1) \leq 2^{2\mu},\qquad \widetilde{N}_{u_\infty} (1/2) = \lim_k \widetilde{N}_{\widetilde{u}_k} (1/2) \geq 2^{2\mu}.
	\end{equation*}
	From Lemma \ref{lem-230225-1124}, this implies $\widetilde{N}_{u_\infty}(r) \equiv 4^m$ for some $m\in \{m_j\}_j$, and in particular $\mu = m$. But we have assumed $\mu \notin \{m_j\}$, which is a contradiction. Hence, the claim is proved.
\end{proof}

Using the Claim, now we prove \eqref{eqn-230226-1203}. Let $N_\infty = \liminf N_u(r)$. Since $N_\infty < \infty$, we can find a sequence of $r_k \rightarrow 0$ such that
$N_u (r_k) \leq 2 N_\infty$ for each $k$. 
Using Lemma \ref{lem-230624-1045}, we can deduce that $\widetilde{N}_u(r_k) \leq C N_u (r_k) \leq 2 C N_\infty$.
Now, find a sufficiently small $\epsi>0$ and a characteristic constant $m_N$, such that $2C N_\infty + \epsi \in (2^{2 m_N}, 2^{2 m_{N+1}})$. Fix an $r_k < r_0$ with $r_0$ given in the claim. Applying the claim with $\mu = \log_4 (2C N_\infty + \epsi)$ iteratively, we obtain that for all $j \geq 0$,	$\widetilde{N}_u (2^{-j}r_k) \leq 2 C N_\infty$. 
Finally, for all sufficiently small $r$, we can find some $j$ such that $r\in (2^{-j-2}r_k, 2^{-j-1}r_k)$. Using the second inequality in Lemma \ref{lem-230624-1045}, we obtain $N_u(r) \leq C \Pi_{i=0}^3 \widetilde{N}_u (2^{-i}r_k) \leq C N_\infty^4$.

\subsection{Step 2}

We now are in a position to perform a more precise version of the argument in Step 1, this time using $N$ in place of $\tilde{N}$. Step 1 is used to improve compactness for the less well-behaved quantity $N$.

\begin{lemma} \label{lem-230226-1016}
	Suppose $\liminf_{r\rightarrow 0} N_u (r) < \infty$.
	Then for any $\mu \notin \{m_j\}$, there exists some $r_0 = r_0 (d, \mu, \Omega, u)$, such that $N_u (r) \leq 4^{2\mu}$ implies $N_u (\tau r) \leq 4^{2\mu}$ for any $r<r_0$ and any $\tau \in [1/16,1/4]$.
\end{lemma}
\begin{proof}
	We prove by contradiction. Suppose the contrary that there exist sequences $r_k \rightarrow 0$ and $\tau_k \in [1/16,1/4]$, such that
	\begin{equation} \label{eqn-230226-0630-1}
		N_{u}(r_k) \leq  4^{2\mu}, \quad N_{u}(\tau_k r_k) >  4^{2\mu}.
	\end{equation}
	Recall from Step 1, for all large enough $k$, 
	\begin{equation} \label{eqn-230226-0630-2}
		N_u(4r_k) \leq 2\limsup_{r\rightarrow 0} N_u (r) \leq 2 C (\liminf_{r\rightarrow 0} N_u (r))^4 = 2CM^4,
	\end{equation}
	where we denote $M:= \liminf_{r\rightarrow 0} N_u (r)$. Let
	\begin{equation*}
		u_k (y) := u(r_k y) \big/ (\fint_{\p B_{r_k /4}} |u|^2)^{1/2}.
	\end{equation*}
	Then
	\begin{equation*} 
		\fint_{\p B_{1/4}} |u_k|^2 = 1, \quad
		\begin{cases}
			\Delta u_k = 0 \quad &\text{in}\,\,(\Omega/r_k) \cap B_4,\\
			u_k = 0 \quad &\text{on}\,\, B_4 \setminus (\Omega /r_k).
		\end{cases}
	\end{equation*}
Combining with \eqref{eqn-230226-0630-1} and $\Delta (u_k)^2 \geq 0$, we obtain $\fint_{\p B_1} |u_k|^2 \leq 4^{2\mu} \fint_{\p B_{1/4}} |u_k|^2 = 4^{2\mu}$. Hence, also noting \eqref{eqn-230226-0630-2}, we reach
	\begin{equation} \label{eqn-230226-0958}
		\fint_{B_4} |u_k|^2 \leq \fint_{\p B_4} |u_k|^2 \leq 2 C M^4 \fint_{\p B_1} |u_k|^2 \leq 2 C M^4 4^{2\mu}.
	\end{equation}
	From the Caccioppoli inequality and the Sobolev embedding, passing to a subsequence, for all $\epsi \in (0,4)$,
	\begin{equation*}
		u_k \rightarrow u_\infty \quad \text{weakly in} \,\, L^2(B_4), H^1(B_{4-\epsi}), \,\,\text{strongly in}\,\, L^2(B_{4-\epsi}).
	\end{equation*}
	Passing to further subsequences, we can also require that $\tau_k \rightarrow \tau_\infty \in [1/16, 1/4]$ and
	\begin{equation} \label{eqn-230226-0937}
		u_k \rightarrow u_\infty \quad \text{strongly in} \,\, L^2(\p B_{\tau_\infty}), L^2(\p B_{\tau_\infty/4}), L^2(\p B_{1}), L^2(\p B_{1/4}).
	\end{equation}
Hence,
	\begin{equation} \label{eqn-230226-1012-1}
		N_{u_\infty}(1) 
		= 
		\frac{\fint_{\p B_1} |u_\infty|^2}{\fint_{\p B_{1/4}} |u_\infty|^2} 
		= 
		\lim_{k\rightarrow \infty} \frac{\fint_{\p B_1} |u_k|^2}{\fint_{\p B_{1/4}} |u_k|^2} 
		= 
		\lim_{k\rightarrow \infty} N_{u} (r_k)
		\leq
		4^{2\mu}.
	\end{equation}
Here, in the last inequality we used  \eqref{eqn-230226-0630-1}.
	Next, we show
	\begin{equation} \label{eqn-230226-1007}
		\fint_{\p B_{\tau_k}} |u_k|^2 \rightarrow \fint_{\p B_{\tau_\infty}} |u_\infty|^2 \quad \text{and}\quad \fint_{\p B_{\tau_k/4}} |u_k|^2 \rightarrow \fint_{\p B_{\tau_\infty/4}} |u_\infty|^2.
	\end{equation}
	For the first limit, we estimate
	\begin{align}
		\left| \fint_{\p B_{\tau_k}} |u_k|^2 - \fint_{\p B_{\tau_\infty}} |u_k|^2 \right|
		&=
		\left| \fint_{\p B_1} \left( |u_k(\tau_k x )|^2 - |u_k(\tau_\infty x)|^2 \right) \, d\sigma_x \right|	\nonumber
		\\&=
		\left| \fint_{\p B_1} \int_{\tau_\infty}^{\tau_k} \frac{d}{dr}|u_k(r x)|^2 \,dr d\sigma_x  \right|	\nonumber
		\\&\leq
		C\|u_k\|_{L^\infty(B_{1/4})} \left| \fint_{\p B_1} \int_{\tau_\infty}^{\tau_k} |\nabla u_k(r x)| \,dr d\sigma_x  \right|
		. \label{eqn-230621-0101}
	\end{align}
	From the mean value property for $u_k^2$ (which is subharmonic) and \eqref{eqn-230226-0958}, we have $\|u_k\|_{L^\infty(B_{1/4})} \leq C_\mu$. Hence,
	\begin{align*}
	\text{RHS of \eqref{eqn-230621-0101}}
		&\leq
		C \left| \fint_{\p B_1} \int_{\tau_\infty}^{\tau_k} |\nabla u_k(r x)| \,dr d\sigma_x  \right|
		\\&\leq
		C \left| \int_{\tau_\infty}^{\tau_k} \fint_{\p B_r}  |\nabla u_k| \,d\sigma dr   \right|
		\leq
		C \|\nabla u_k\|_{L^2(B_{1/4})}\sqrt{|\tau_k - \tau_\infty|} \rightarrow 0.
	\end{align*}
	The last step used that $\nabla u_k$ is uniformly bounded in $H^1(B_1)$. Combining with \eqref{eqn-230226-0937}, we have
	\begin{equation*}
		\left|\fint_{\p B_{\tau_k}} |u_k|^2 - \fint_{\p B_{\tau_\infty}} |u_\infty|^2\right|
		\leq
		\left| \fint_{\p B_{\tau_k}} |u_k|^2 - \fint_{\p B_{\tau_\infty}} |u_k|^2 \right| + \left| \fint_{\p B_{\tau_\infty}} |u_k|^2 - \fint_{\p B_{\tau_\infty}} |u_\infty|^2 \right|
		\rightarrow
		0.
	\end{equation*}
	This proves the first convergence in \eqref{eqn-230226-1007}. The proof for the second convergence is almost identical. From \eqref{eqn-230226-1007} and \eqref{eqn-230226-0630-1},
	\begin{align} \label{eqn-230226-1012-2}
		N_{u_\infty} (\tau_\infty) 
		= 
		\frac{ \fint_{\p B_{\tau_\infty}} |u_\infty|^2 }{ \fint_{\p B_{\tau_\infty/4}} |u_\infty|^2 }
		=
		\lim_{k\rightarrow \infty} \frac{ \fint_{\p B_{\tau_k}} |u_k|^2 }{ \fint_{\p B_{\tau_k/4}} |u_k|^2 }
		=
		\lim_{k\rightarrow \infty} N_u (\tau_k r_k)
		\geq
		4^{2\mu}.
	\end{align}
	As before, $u_\infty$ satisfies \eqref{eqn230226-0924}. From \eqref{eqn-230226-1012-1}, \eqref{eqn-230226-1012-2}, and the rigidity in Lemma \ref{lem-230225-1124}, we must have $N_{u_\infty} \equiv 16^m$ for some $m \in \{m_j\}$. Hence, $\mu = m$, a contradiction.
\end{proof}

\subsection{Step 3: Conclusion of the proof of Theorem \ref{thm-230225-1127}}

We may as well assume that $\liminf N_u(r) < + \infty$. Recall that $N_u(r) \geq 1$ for all $r$, and let $m := 2^{-1} \log_4 (\liminf N_u(r)).$ We have $m \in [0,\infty)$.
	
Now, we find a sequence of positive numbers $\epsi_k \rightarrow 0$, such that $m + \epsi_k \notin \{m_j\}$. For each $k$, we further find a small enough $r_k$ with $N_u(r_k) < m + \epsi_k$ and $r_k < r_0(d,m + \epsi_k, \Omega,u)$ (where $r_0$ is given in Lemma \ref{lem-230226-1016}). Applying Lemma \ref{lem-230226-1016} iteratively, we have $sup_{r \leq r_k/4} N_u(r) \leq 4^{2 (m + \epsi_k)}$, and so in particular $\limsup_{r \rightarrow 0} N_u(r) \leq 4^{2 (m + \epsi_k)}$. Sending $k\rightarrow \infty$,
\begin{equation*}
	\limsup_{r\rightarrow 0} N_u(r) \leq 4^{2 m} = \liminf N_u(r).
\end{equation*}
This implies, passing to a subsequence, $u_r = u(r\cdot)/(\fint_{\p B_r} |u|^2)^{1/2}$ converges to a non-trivial, homogeneous harmonic function on $B_1 \cap \Gamma$, with the homogeneity $m$. This implies $m$ must be one of the characteristic constants defined in \eqref{eqn-221018-0432}.

\section{Eigenvalues, eigenfunctions, and Green's functions on cones} \label{sec-230703-0347}
Before proving Theorem \ref{thm-221218-1120}. we make some preparatory remarks concerning the Green's function on the limit cone.  Let $\Gamma\subset \bR^d$ be a cone with vertex at the origin and $\Sigma = \Gamma \cap \p B_1$ be its spherical cross-section. 

\subsection{Eigenvalues and eigenfunctions of spherical cross-sections}

Let 
\begin{equation*}
	\lambda_1 \leq \lambda_2 \leq \cdots \lambda_k \leq \cdots
\end{equation*}
be the Dirichlet eigenvalues (counting multiplicity) of the spherical cross-section $\Sigma$ and $\{\psi_k\}_{k=1}^\infty$ be a corresponding basis of eigenfunctions, orthonormal in $L^2$.
We have the following properties.
\begin{lemma}\label{lem-221220-1021}
	For each $q > (d-1)/4$, there exists $C = C(q, d, \Sigma)>0$ independent of $k$, such that
	\begin{equation*}
		\| \psi_k \|_{L^\infty(\Sigma)} \leq C \lambda_k^q.
	\end{equation*}
\end{lemma}

\begin{proof}
	Applying the local maximum principle to $v = |\psi_k|$, which is a weak subsolution of $- \Delta_{S^{n-1}} v \leq \lambda_k v$, in charts (see e.g. \cite{MR1814364}[Theorem 8.17]),
	\[
		\|\psi_k\|_{L^\infty(\Sigma)} \leq C [\|\lambda_k  \psi_k\|_{L^{2q}} + \|\psi_k\|_{L^2}]
	\]
	noting $2q > (d-1)/2$. Then
	\[
		\|\lambda_k  \psi_k\|_{L^{2q}} \leq \lambda_k \|\psi_k\|_{L^\infty}^{\frac{q - 1}{q}}\|\psi_k\|_{L^2}^{\frac{1}{q}} \leq \epsilon \|\psi_k\|_{L^\infty} + C_\epsilon \lambda_k^{q},
	\]
	using that $\|\psi_k\|_{L^2} = 1$. Choosing $\epsilon$ small and reabsorbing the first term gives
	\[
		\|\psi_k\|_{L^\infty(\Sigma)} \leq C \lambda_k^{q}.
	\]
\end{proof}

\begin{lemma} \label{lem-230216-0420}
	For some $C=C(d)$, $\lambda_k \geq \frac{1}{C}k^{2/(d-1)}$.
\end{lemma}
\begin{proof}
	Since $\Sigma \subset \p B_1$, we have $\lambda_k \geq \lambda_k (S^{d-1})$,
	where $\lambda_k(S^{d-1})$ is the $k$th eigenvalue of Laplacian operator on a $(d-1)$-dimensional unit sphere. We know that $\{\lambda_k(S^{d-1})\}_k$ contains
	\begin{equation*}
		\text{exactly one zero and}\,\,
		\begin{pmatrix}
			d-1+j \\ j
		\end{pmatrix}
		-
		\begin{pmatrix}
			d+j-2 \\ j-1
		\end{pmatrix}
		\,\,\text{copies of}\,\,
		j(j+d-2), \quad j = 1,2,\ldots .
	\end{equation*}
	By a calculation and Stirling's formula,
	\begin{align*}
		\begin{pmatrix}
			d-1+j \\ j
		\end{pmatrix} 
		-
		\begin{pmatrix}
			d+j-2 \\ j-1
		\end{pmatrix}
		&=
		\frac{(j+ d-2)!}{j! (d-2)!}
		\\&\approx
		\frac{(j+d-2)^{j+d-2} \sqrt{2\pi (j+d-2)}}{e^{j+d-2}} \big/ \frac{j^j \sqrt{2\pi j}}{e^j}
		\approx
		j^{d-2}.
	\end{align*}
	Hence, counting all eigenvalues up to the size $j(j+d-2) \approx j^2$, we reach
	\begin{equation*}
		\lambda_k \geq \frac{1}{C}\left( \left( \frac{k}{C} \right)^{1/(d-1)} - 1 \right)^2 \geq \frac{1}{C} k^{2/(d-1)}.
	\end{equation*}
\end{proof}
As a direct corollary of Lemma \ref{lem-230216-0420},
\begin{equation} \label{eqn-230202-0546}
	\sum_{j=1}^{\infty} \mu^{ \sqrt{\lambda_k} } < \infty, \quad \forall \mu\in (0,1).
\end{equation}

\subsection{Green's function on a cone and orthogonal expansions} \label{sec-221228-1031}

From standard elliptic regularity theory, the Green's function $G(x,y)$ exists on an arbitrary Lipschitz cone $\Gamma$. More precisely, for every $x \in \Gamma\cap B_1$, $f, g_i \in L^\infty (\Gamma \cap B_1)$, and $h \in C^0_c (\p\Gamma \cap B_1)$,
the unique continuous weak solution to
\begin{equation*}
	\begin{cases}
		\Delta u = f + \partial_i g_i \quad \text{in}\,\,\Gamma \cap B_1,\\
		u = h \quad \text{on} \,\, \p(\Gamma \cap B_1),\\
		u\rightarrow 0 \quad \text{as} \,\, |x|\rightarrow \infty
	\end{cases}
\end{equation*}
can be represented by
\begin{equation*}
	u(x) = \int_\Gamma G(x,y) f(y)\,dy - \int_\Gamma \frac{\p}{\p y_i} G(x,y) g_i(y)\,dy + \int_{\p\Gamma} h(y) \frac{\p}{\p \vec{n_y}} G(x,y) \,d\sigma_y.
\end{equation*}
See, for instance \cite[Theorem~1.1]{MR657523} and \cite{MR1282720}. See also \cite{MR2341783} for discussions on unbounded domains.
The following properties are standard: symmetry $G(x,y) = G(y,x)$, scaling $G(\lambda x, \lambda y) = \lambda^{2-d} G(x,y)$, and a pointwise bound
\begin{equation}\label{eqn-221218-1113}
	G(x,y) \leq C \frac{1}{|x-y|^{d-2}}.
\end{equation} 
Furthermore, we have the following derivative bounds when $\Gamma$ is regular enough.
\begin{lemma} \label{lem-221227-0652}
	Suppose that $\Gamma$ satisfies Assumption \ref{ass-230119-0339}. Then for $x,y \in \Gamma, x\neq y$, we have
	\begin{equation} \label{eqn-221227-0835}
		|\nabla_y G(x,y)| \leq C ( |y|^{-1} + |x-y|^{-1} ) |x-y|^{2-d},
	\end{equation}
	\begin{equation} \label{eqn-221227-0816-1}
		|\nabla_y G(x,y)| \leq C \delta(x) ( |x|^{-1} + |x-y|^{-1} ) ( |y|^{-1} + |x-y|^{-1} ) |x-y|^{2-d},
	\end{equation}
	where $\delta(x) = \dist(x,\p\Gamma)$.
\end{lemma}
The proof is standard, which is based on scaling, the point-wise bound \eqref{eqn-221218-1113}, and a local Lipschitz estimate coming from the smoothness of $\Sigma$. For completeness, we provide a proof in Appendix \ref{sec-230315-0811}.
Throughout the rest of the paper paper, we denote
\begin{equation*}
	K_i (x,y) := \frac{\p G(x,y)}{\p y_i} \,\, \text{for} \,\, x,y\in\Gamma, x\neq y \quad\text{and} \quad  k(x,y) := \frac{\p G (x,y)}{\p \vec{n}_y} \,\, \text{for} \,\, x \in \Gamma, y \in \p\Gamma \setminus \{0\}.
\end{equation*}
From Lemma \ref{lem-221227-0652}, we know that $K_i(\cdot,y), k(\cdot, y) \in L^\infty_{loc} (B_{|y|})$ are harmonic functions, which have orthogonal expansions
\begin{equation*}
	K_i (x,y)
	=
	\sum_{j=1}^\infty b_i^{(j)} (y) |x|^{m_j} \psi_j \left( x / |x| \right),
	\quad
	K (x,y)
	=
	\sum_{j=1}^\infty b^{(j)} (y) |x|^{m_j} \psi_j \left( x / |x| \right),
\end{equation*}
where
\begin{equation*}
	b_i^{(j)} (y)
	=
	\frac{ \int_{\p B_{2|y|/3} \cap \Gamma} K_i(z,y)|z|^{m_j} \psi_j(z/|z|) \,dz }{ \int_{\p B_{2|y|/3} \cap \Gamma} \left( |z|^{m_j} \psi_j(z/|z|) \right)^2 \,dz },
	b^{(j)} (y)
	=
	\frac{ \int_{\p B_{2|y|/3} \cap \Gamma} k(z,y)|z|^{m_j} \psi_j(z/|z|) \,dz }{ \int_{\p B_{2|y|/3} \cap \Gamma} \left( |z|^{m_j} \psi_j(z/|z|) \right)^2 \,dz }.
\end{equation*}
Here, $2/3$ could have been any fixed number smaller than $1$. By scaling,
\begin{equation} \label{eqn-230518-1255}
	b_i^{(j)} (y) = |y|^{1-d-m_j} b_i^{(j)} (y/|y|),
	\quad
	b^{(j)} (y) = |y|^{1-d-m_j} b^{(j)} (y/|y|).
\end{equation}
Denote partial sums as
\begin{equation} \label{eqn-230626-1015}
	\begin{split}
	K_i^{(N)} (x,y) = \sum_{j\leq N} b_i^{(j)}\left( y / |y| \right) |y|^{1-d-m_j} |x|^{m_j} \psi_j \left( x / |x| \right),\\
	k^{(N)} (x,y) = \sum_{j\leq N} b^{(j)}\left( y / |y| \right) |y|^{1-d-m_j} |x|^{m_j} \psi_j \left( x / |x| \right).
	\end{split}
\end{equation}
We have the following estimates.
\begin{lemma} 
	Suppose that $\Gamma$ satisfies Assumption \ref{ass-230119-0339} and $b_i^{(j)}, b^{(j)}$ are defined as above. Then for some $C=C(d,\Sigma)$,
	\begin{equation} \label{eqn-230315-1042}
		\left| b_i^{(j)}\left( y / |y| \right) \right| + \left| b^{(j)}\left( y / |y| \right) \right| 
		\leq 
		C (\lambda_j)^{\frac{d}{4}} \left(3 / 2\right)^{m_j}, \quad j= 1, 2, \ldots.
	\end{equation}
Furthermore, we have remainder estimates: for some $C=C(d,\Sigma, N)$,
	\begin{equation} \label{eqn-221221-0701}
		\begin{split}
		|K_i (x,y) - K_i^{(N)} (x,y)| \leq C \frac{|x|^{m_{N+1}}}{|y|^{m_{N+1} + d - 1}}, \quad \forall y \in \Gamma \setminus \{0\}, x \in \Gamma \cap B_{|y|/2}, \\
		|k (x,y) - k^{(N)} (x,y)| \leq C \frac{|x|^{m_{N+1}}}{|y|^{m_{N+1} + d - 1}}, \quad \forall y \in \p\Gamma \setminus \{0\}, x \in \Gamma \cap B_{|y|/2}.
		\end{split}
	\end{equation}
\end{lemma}
Here, \eqref{eqn-230315-1042} is not sharp in general, but as it is enough for later use we do not pursue more precise bounds.
\begin{proof}
	We only prove for $K_i$ and $b_i^{(j)}$ as the computation for $k$ and $b^{(j)}$ is almost identical. For \eqref{eqn-230315-1042}, by the scaling property \eqref{eqn-230518-1255} and  $\|\psi_j\|_{L^2} = 1$,
	\begin{align*}
		\left| b_i^{(j)} \left(y / |y|\right) \right|
		&=
		(3/2)^{-1+d+m_j} \left| b_i^{(j)} \left(3y/ (2|y|)\right) \right| \nonumber
		\\&=
		(3/2)^{-1+d+m_j} \left| \int_{\Gamma \cap \p B_1} K_i( w, 3y / (2|y|) ) \psi_j (w) \,dw \right| \nonumber
		\\&\leq
		(3/2)^{-1+d+m_j} C (1/2)^{1-d} \|\psi_j\|_{L^\infty} \nonumber
		\\&\leq
		C \lambda_j^{d/4} (3/2)^{m_j}. 
	\end{align*}
	Here we have also used the point-wise bound \eqref{eqn-221227-0835} and Lemma \ref{lem-221220-1021} with $q=d/4$. Next, we prove \eqref{eqn-221221-0701}:
	\begin{align*}
		|K_i(x,y) - K_i^{(N)} (x,y)| 
		&= 
		\left| \sum_{j=N+1}^\infty b_i^{(j)} \left( y / |y| \right) |y|^{1-d-m_j} |x|^{m_j} \psi_j \left( x/|x| \right) \right|
		\\&\leq
		\frac{|x|^{m_{N+1}}}{|y|^{m_{N+1}+d-1}} \sum_{j=N+1}^\infty \left| b_i^{(j)} \left( y / |y| \right) \right| \left( |x| / |y| \right)^{m_j - m_{N+1}} \| \psi_j \|_{L^\infty}
		\\&\leq
		C \frac{|x|^{m_{N+1}}}{|y|^{m_{N+1}+d-1}} \sum_{j=N+1}^\infty \left( 3 / 2 \right)^{m_j} \left(1/2\right)^{m_j - m_{N+1}} \lambda_j^{d/2}.
	\end{align*}
	Here in the last inequality, we used \eqref{eqn-230315-1042}, $|x| < |y|/2$, and Lemma \ref{lem-221220-1021}. Hence,
	\begin{align*}
		|K_i(x,y) - K_i^{(N)} (x,y)|
		&\leq
		C \frac{|x|^{m_{N+1}}}{|y|^{m_{N+1}+d-1}} \sum_{j=N+1}^\infty \left( 3/ 4 \right)^{m_j} \lambda_j^{d/2} \nonumber
		\\&\leq
		C \frac{|x|^{m_{N+1}}}{|y|^{m_{N+1}+d-1}} \sum_{j=N+1}^\infty \left( 3 / 4 \right)^{m_j - (d/2)\log_{4/3}(\lambda_j)} \nonumber
		\\& \leq
		C \frac{|x|^{m_{N+1}}}{|y|^{m_{N+1}+d-1}} \sum_{j=N+1}^\infty \left( 3 / 4 \right)^{ \sqrt{\lambda_j}/2 } 
		\\&\leq
		C \frac{|x|^{m_{N+1}}}{|y|^{m_{N+1}+d-1}} \sum_{j=1}^\infty \left( 3 / 4 \right)^{ \sqrt{\lambda_j}/2 }
		=
		C \frac{|x|^{m_{N+1}}}{|y|^{m_{N+1}+d-1}}. \nonumber
	\end{align*}
	Here we have also used \eqref{eqn-221018-0432} and \eqref{eqn-230202-0546}. The lemma is proved.
\end{proof}

\section{Proof of Theorem \ref{thm-221218-1120}} \label{sec-230703-0348}

In this section, we prove Theorem \ref{thm-221218-1120}. 
From Assumption \ref{ass-230119-0339}, we fix coordinates $x=(x',x_d)$ such that the tangent cone $\Gamma$ at the origin can be locally represented by $\{x_d > \Psi(x') \}$ for some $1$-homogeneous function $\Psi$. By \eqref{eqn-220728-0546}, we know that for sufficiently large $C$, locally 
\begin{equation} \label{eqn-230707-1136}
	\mathcal{C} := \{(x',x_d) : x_d > \Psi(x') + C|x'|^{1+\alpha}\} \subset \Omega.
\end{equation}
The following De Giorgi-type estimate plays a key role in our proof.
\begin{lemma} \label{lem-230119-0733}
	Assume $\Omega$ together with its tangent cone at the origin $\Gamma$ satisfy assumptions of Theorem \ref{thm-221218-1120}, and $\mathcal{C}$ be defined as above. Let $u\in H^1_{loc}$ satisfy
	\begin{equation}\label{eqn-220729-1048}
		\partial_i(a_{ij}\partial_j u) = 0 \,\, \text{in}\,\,\Omega\cap B_2,
		\quad
		u = 0 \,\, \text{on}\,\,\p\Omega\cap B_2,
	\end{equation}
where for some $\lambda, \Lambda>0$,
\begin{equation} \label{eqn-221221-0751-1}
	a_{ij}\xi_i\xi_j \geq \lambda |\xi|^2, \forall \xi \in \bR^d\setminus \{0\},
	\quad
	|a_{ij}| \leq \Lambda.
\end{equation}
Then there exists a constant $\alpha_0 \in (0,1)$, such that for all small enough $r>0$,
	\begin{equation*} 
		\sup_{B_r \cap (\Omega \setminus \mathcal{C})} |u| \leq C r^{\alpha \alpha_0} \sup_{B_{2r}\cap \Omega} |u|.
	\end{equation*}
\end{lemma}
\begin{proof}
	Fix a point $x \in B_r \cap (\Omega \setminus \mathcal{C})$.
	From the definition of $\mathcal{C}$ and \eqref{eqn-220728-0546}, there exists a constant $C_1>0$, such that for all sufficiently small $r$, we have $\dist(\p\Omega\cap B_r, \p \mathcal{C} \cap B_r) \leq C_1 r^{1+\alpha}$ and
	\begin{equation} \label{eqn-230119-0728}
		|B_s (x) \cap \Omega^c| \geq c_0 |B_s| \quad 
		\forall  s \in (10 C_1 r^{1+\alpha}, r/2).
	\end{equation}
	Now, from $u=0$ on $\p\Omega$, $B_{10 C_1 r^{1+\alpha}}(x) \cap \p\Omega \neq \emptyset$, and the De Giorgi improvement of oscillation lemma, for some $\alpha_0\in(0,1)$,
	\begin{equation*}
		|u(x) - 0| \leq \operatorname{OSC}_{\Omega \cap B_{10 C_1 r^{1+\alpha}} (x) }(u) \leq C \left(\frac{10 C_1 r^{1+\alpha}}{r/2}\right)^{\alpha_0} \sup_{B_{r/2}(x)} |u| \leq C r^{\alpha \alpha_0} \sup_{B_{2r}} |u|.
	\end{equation*}
	See for instance \cite[Theorem~8.27]{MR1814364}. Note that in the proof of \cite[Theorem~8.27]{MR1814364}, the exterior cone condition can be replaced by the exterior measure condition. Moreover, here we only need the estimate up to the scale $C_1 r^{1+\alpha}$ instead of zero, and at these scales the exterior measure condition is true due to \eqref{eqn-230119-0728}.
\end{proof}

\begin{proposition} \label{prop-221228-0952}
	Let $\Omega$ and $\Gamma$ satisfy assumptions of Theorem \ref{thm-221218-1120}, and in addition, for some small $R>0$, $\Gamma \cap B_R \subset \Omega \cap B_R$. Let $m_N < m_{N+1}$ be two distinct characteristic constants of $\Gamma$. Suppose that $u \in H^1_{loc}$ satisfies \eqref{eqn-220729-1048} with coefficients satisfying \eqref{eqn-221221-0751-1} and
	\begin{equation} \label{eqn-221221-0751}
		|a_{ij}(x) - \delta_{ij}| = O(|x|^\beta),\quad \text{as} \,\, x \rightarrow 0
	\end{equation}
	for some $\beta > 0$. 
	Then, if $u(x) = O(|x|^{\mu})$ for some $\mu \in [m_N, m_{N+1})$, we have
	\begin{equation} \label{eqn-230709-0856}
		u (x) = C |x|^{m_N} \psi_N \left( \frac{x}{|x|} \right) + w(x)\,\,\text{in}\,\,\Gamma,
		\,\,
		\text{with}
		\,\,
		\left( \fint_{B_r \cap \Gamma} |w|^p \right)^{\frac{1}{p}} = O(r^{\min\{ \mu + \alpha_0 \min \{\alpha,\beta\}, m_{N+1} \} }),
	\end{equation}
	where $p> 2d/(d-2)$ and $\alpha_0 \in (0,1)$ are constants depending only on $(\Omega, d, \lambda, \Lambda)$. Furthermore, when $\mu \in (m_N, m_{N+1})$, we must have $C \equiv 0$.
\end{proposition}
Assuming Proposition \ref{prop-221228-0952}, we prove Theorem \ref{thm-221218-1120}
\begin{proof}[Proof of Theorem \ref{thm-221218-1120}]
	Clearly we only need to consider the case of $\lim_{r\rightarrow 0} N_u < \infty$. We claim that there exists a characteristic constant $m_N$ such that
	\begin{equation} \label{eqn-230719-1247}
		|u (x)| = O(|x|^{m_N}) \quad \text{and} \quad |u (x)| \neq O(|x|^{m_{N+1}}).
	\end{equation}
	Indeed, if $N_u(r) \leq M$, we have that $\fint_{B_{2^{-k}}}u^2 \geq c M^{-k}$ from Lemma \ref{lem-230624-1045}, and so $\|u\|_{L^\infty(B_r)}^2 \geq \fint_{B_{r}}u^2 \geq c r^\beta$ for some $\beta > 0$. This shows that $|u| \neq  O(|x|^\mu)$ for some $\mu < \infty$. We still need to show $u = O(|x|^{m_1})$, which can be done by constructing a barrier function. By \eqref{eqn-220728-0546}, we know that for sufficiently large $C$, locally 
	\begin{equation*}
		\Omega \subset \mathcal{C}_1 := \{(x',x_d) : x_d > \Psi(x') - C|x'|^{1+\alpha}\}.
	\end{equation*}
Now for sufficiently small $R>0$, let $u_1 \in H^1_{loc}$ be the solution to
\begin{equation*}
	\Delta u_1 = 0\,\,\text{in}\,\, \mathcal{C}_1 \cap B_R,
	\quad u_1 = 0 \,\,\text{on}\,\,\p \mathcal{C}_1 \cap B_R,
	\quad
	u_1 = 1 \,\,\text{on}\,\, \mathcal{C}_1 \cap \p B_R.
\end{equation*}
	Now, $\mathcal{C}_1$ is a ``$C^{1,\alpha}$ perturbation'' of the cone $\Gamma$, which verifies the assumptions in \cite[Theorem~1.1]{MR3109767}. Hence, for some non-trivial homogeneous harmonic function $P_{m_1}$ of degree $m_1$, we have $(0<) u_1 = P_{m_1}(x) + o(|x|^{m_1})$, noting that $m_1$ is the characteristic constant associated with the leading eigenvalue of the tangent cone $\Gamma$. Now, since $\Omega \cap B_R \subset \mathcal{C}_1 \cap B_R$ by choosing $R$ small enough, by comparison, we have $|u| \leq u_1 = O(|x|^{m_1})$. Combining these, we have proved \eqref{eqn-230719-1247}.

Now we prove \eqref{eqn-230709-0856}. Take a $C^{1,\alpha}$ change of variables
\begin{equation*}
	(\widetilde{x}', \widetilde{x}_d) = \Phi (x',x_d) = (x', x_d - C|x'|^{1+\alpha}),
\end{equation*}
	where $C$ is the number given in \eqref{eqn-230707-1136}. It is easy to see that $\widetilde{\Omega} := \Phi (\Omega)$ is still $\alpha$-conical with the tangent cone $\widetilde{\Gamma} := \Phi (\mathcal{C})$, which satisfies Assumption \ref{ass-230119-0339}. More importantly, now the tangent cone $\widetilde{\Gamma}$ is locally contained in $\widetilde{\Omega}$. 
	
	In new coordinates, $\widetilde{u} := u \circ \Phi^{-1}$ satisfies \eqref{eqn-220729-1048} locally on $\widetilde{\Omega}$, with coefficients verifying all the conditions in Proposition \ref{prop-221228-0952}. Moreover, we still have $\widetilde{u} = O(|\widetilde{x}|^{m_N})$ since $\Phi$ is locally a diffeomorphism. Applying Proposition \ref{prop-221228-0952} with $\mu = m_N$ to $\widetilde{u}$ on $\widetilde{\Omega}$, we obtain a homogeneous harmonic function $P (\widetilde{x}) = C_1 |\widetilde{x}|^{m_N} \psi_N (\widetilde{x} / |\widetilde{x}|)$ on $\widetilde{\Gamma}$, such that 
	\begin{equation} \label{eqn-230709-1149}
		\left( \fint_{\widetilde{\Gamma} \cap B_r} |\widetilde{u}(\widetilde{x}) - P (\widetilde{x}) |^p \,d\widetilde{x}\right)^{1/p}  \leq C r^{\min\{ m_N + \alpha_0 \alpha, m_{N+1} \} }.
	\end{equation} 
	We first show the leading term $P \not\equiv 0$, i.e., $C_1 \neq 0$.
	Suppose the contrary that $C_1=0$. Again from the fact that $\Phi$ is locally a diffeomorphism, by Lemma \ref{lem-230119-0733}, we have 
	\begin{equation*}
		\| \widetilde{u} \|_{L^\infty (B_r \cap (\widetilde\Omega \setminus \widetilde{\Gamma}))} 
		\leq 
		C \|u\|_{L^\infty (B_{C r} \cap (\Omega \setminus \mathcal{C}))}
		\leq
		C r^{\alpha \alpha_0} \sup_{B_{2Cr}} |u|
		\leq
		C r^{m_N + \alpha \alpha_0}.
	\end{equation*}
	Combining this, \eqref{eqn-230709-1149}, and a local maximum principle, we reach $|u(x)| = O(|x|^{\mu_1})$, where $\mu_1 := \min\{ m_N + \alpha_0 \alpha, m_{N+1} \}$. If $\mu_1 = m_{N+1}$, this is a contradiction. Otherwise, repeating the above procedure, but now applying Proposition \ref{prop-221228-0952} with $\mu = \mu_1$, we can further improve the vanishing order of $u$, and in finitely many steps reach $|u| = O(|x|^{m_{N+1}})$, which is again a contradiction. Hence, $C_1 \neq 0$.
	
	Set $w(x) = u(x) - P(x)$, where $u$ and $P$ are extended by zero outside $\Omega$ and $\Gamma$, respectively. We are left to show $\|w\|_{L^\infty (B_r)} \leq C r^{\min\{ m_N + \alpha_0 \alpha, m_{N+1} \} }$.	
	Note that both $u$ and $P$ satisfy assumptions of Lemma \ref{lem-230119-0733} (for $P$, we take $\Omega = \Gamma$). Hence, $\sup_{B_r \setminus \mathcal{C}} |w| \leq \sup_{B_r \setminus \mathcal{C}} ( |u| + |P|) \leq C r^{m_N + \alpha \alpha_0}$. Here, we also used $u=0$ on $B_r \setminus \Omega$ and $P=0$ on $B_r \setminus \Gamma$. To bound $w$ on $B_r \cap \mathcal{C}$, we transforming \eqref{eqn-230709-1149} back to $x$-coordinates and apply the triangle inequality. This gives
	\begin{align*}
		&\left( \fint_{\mathcal{C} \cap B_r} | u (x) - P (x) |^p \, dx \right)^{1/p}
		\\&\quad\leq
		\left( \fint_{\mathcal{C} \cap B_r} | u (x) - P \circ \Phi (x) |^p \, dx \right)^{1/p}
		+ 
		\left( \fint_{\mathcal{C} \cap B_r} | P(x) - P \circ \Phi (x) |^p \, dx \right)^{1/p}
		\\&\quad\leq
		C \left( \fint_{\widetilde{\Gamma} \cap B_{C r}} | u\circ \Phi^{-1} - P |^p \right)^{1/p}
		+
		\sup_{B_r \cap \mathcal{C}} ( |x-\Phi(x)| |\nabla \Phi| )
		\leq 
		C r^{\min\{ m_N + \alpha_0 \alpha, m_{N+1} \} }.
	\end{align*}
	Here, we also used $|x - \Phi(x)| \leq C|x|^{1+\alpha}$ and $|\nabla P| = O(|x|^{m_N - 1})$. Now, note that
	\begin{equation*}
		\Delta w = \Delta (u - P) = 0\,\,\text{in}\,\,B_1 \cap \mathcal{C},\quad |w| = |u-P| \leq C |x|^{m_N + \alpha \alpha_0}\,\,\text{on} \,\, B_1 \cap \p \mathcal{C}.
	\end{equation*}
	We apply a local maximum principle to $(w - C r^{m_N + \alpha \alpha_0})_+$ on $B_{2r} \cap \mathcal{C}$, which is subharmonic in $B_{2r} \cap \mathcal{C}$ and vanishes on $B_{2r} \cap \p \mathcal{C}$ by choosing $C$ large enough, to obtain
	\begin{equation*}
		\| (w - C r^{m_N + \alpha \alpha_0})_+ \|_{L^\infty(B_{r} \cap \mathcal{C})} \leq C r^{d/2} \| (w - C r^{m_N + \alpha \alpha_0})_+ \|_{L^2(B_{r} \cap \mathcal{C})} \leq C r^{m_N + \alpha \alpha_0}.
	\end{equation*}
	Similarly, we can bound $\|(-w - C r^{m_N + \alpha \alpha_0})_+\|_{L^\infty (B_r \cap \mathcal{C})}$, and hence, $\|w\|_{L^\infty (B_r \cap \mathcal{C})} \leq C r^{m_N + \alpha \alpha_0}$. Combining above, we reach the desired remainder estimate.
\end{proof}
The rest of Section \ref{sec-230703-0348} will be devoted to the proof of Proposition \ref{prop-221228-0952}.

\subsection{A representation by Green's function} 

Let $\eta \in C^\infty_c (B_R)$ be a usual cut-off function with $\eta = 1$ on $B_{R/2}$. Let 
\begin{align*}
	v 
	:=
	- \underbrace{\int_{\Gamma} K_i(x,y) f_i (y)\,dy}_{I} + \underbrace{\int_{\p \Gamma} k( x, y) h( y )\,d\sigma_{y}}_{II},
\end{align*}
where $f_i = \eta (\delta_{ij} - a_{ij}) \p_{x_j} u$, $h = u \eta$, and $K_i$ and $k$ are kernels defined in Section \ref{sec-221228-1031}. Now, $u-v$ satisfies
\begin{equation*}
	\begin{cases}
		\Delta (u-v) = 0 \quad \text{in}\,\, \Gamma \cap B_{R/2},\\
		u-v = 0 \quad \text{on} \,\, \p \Gamma \cap B_{R/2}.
	\end{cases}
\end{equation*}
Hence, we have the full series expansion
\begin{equation*}
	u - v = \sum_{j=1}^\infty C_j |x|^{m_j} \psi_j \left( x/|x| \right) = \sum_{j=1}^N C_j |x|^{m_j} \psi_j \left( x/|x| \right) + O(|x|^{m_{N+1}}).
\end{equation*}
Next, we expand $v$. First, from Lemma \ref{lem-230119-0733}, we have
\begin{equation} \label{eqn-221229-1157}
	\sup_{B_r \cap \p \Gamma} |h| = O(r^{\mu + \alpha \alpha_0}) .
\end{equation}
By \eqref{eqn-221221-0751},
\begin{equation} \label{eqn-221222-1006}
	|f_i| = \eta | \delta_{ij} - a_{ij} | |\p_{x_j} u| \leq C |x|^\beta |\nabla u|.
\end{equation}

\subsection{Estimating the term \texorpdfstring{$I$}{I}} 

Recall the definition for $K_i^{N}$ in \eqref{eqn-230626-1015}. We split
\begin{align*}
	I 
	&= 
	\underbrace{\sum_{i=1}^{d} \int_{\Gamma} K_i^{(N)} (x,y) f_i (y) \,dy}_{I_1}
	-
	\underbrace{\sum_{i=1}^{d} \int_{\Gamma \cap B_{2|x|}} K_i^{(N)} (x,y) f_i (y) \,dy}_{I_2}
	\\&\quad
	+
	\underbrace{\sum_{i=1}^{d} \int_{\Gamma \cap B_{2|x|}^c} (K_i - K_i^{(N)}) (x,y) f_i (y) \,dy}_{I_3}
	+
	\underbrace{\sum_{i=1}^d \int_{\Gamma \cap B_{2|x|}} K_i(x,y) f_i (y) \,dy}_{I_4} .
\end{align*}
In the following,we prove that $I_1$ is a finite combination of homogeneous harmonic functions with degree up to $m_N$, and the rest three terms $I_2,I_3$, and $I_4$ are of higher order.

\noindent\textbf{Convergence and expansion of $I_1$}: for each $j=1,\ldots,N$ and $i=1,\ldots,d$, from \eqref{eqn-230315-1042} and \eqref{eqn-221222-1006},
\begin{equation*}
	\left| b_i^{(j)} \left( y/|y| \right) \right| |y|^{1-d-m_j} |f(y)| \leq C |y|^{1-d-m_j + \beta} |D u (y)|,
\end{equation*}
where $C$ is a constant that could depend on $m_j, \lambda_j$. Hence,
\begin{align}
		\int_{\Gamma \cap B_R} \left| b_i^{(j)}  \left( \frac{y}{|y|} \right) \right| |y|^{1-d-m_j} |f| \, dy
		&\leq
		C \sum_{l=0}^\infty (2^{-l}R)^{1-d-m_j + \beta} (2^{-l}R)^d \fint_{\Gamma \cap (B_{2^{-l}R} \setminus B_{2^{-l-1}R})} |D u| \nonumber
		\\&\leq 
		C \sum_{l=0}^\infty (2^{-l}R)^{1-m_j + \beta} \left( \fint_{\Gamma \cap (B_{2^{-l}R} \setminus B_{2^{-l-1}R})} |D u|^2 \right)^{1/2}. \label{eqn-230626-1019}
\end{align}
By the Caccioppoli inequality, $u= O(|x|^\mu)$, and $m_j \leq \mu$, we can further compute
\begin{equation} \label{eqn-230214-0103}
	\begin{split}
		\text{RHS of \eqref{eqn-230626-1019}}
		&\leq
		C \sum_{l=0}^\infty (2^{-l}R)^{1-m_j + \beta} (2^{-l}R)^{-1} \| u \|_{L^\infty(\Gamma \cap B_{2^{-l+1}R})}
		\\&\leq
		C \sum_{l=0}^\infty (2^{-l}R)^{-m_j + \beta} (2^{-l+1}R)^{ \mu }
		\\&\leq
		C R^{\mu -m_j + \beta}\sum_{l=0}^\infty 2^{-l\beta} 2^{-l(\mu - m_j)} 
		\\&\leq 
		C R^{m_N-m_j + \epsi + \beta} \sum_{l=0}^\infty 2^{-l\beta} \leq C R^{\mu -m_j + \beta} . 
	\end{split}
\end{equation}
This proves that the integrands in $I_1$ are in $L^1$. Hence,
\begin{align*}
	I_1
	&= 
	\sum_{i=1}^d \sum_{j=1}^{N} \int_\Gamma b_i^{(j)} \left( \frac{y}{|y|} \right) |y|^{1-d-m_j} |x|^{m_j} \psi_j \left( \frac{x}{|x|} \right) f_i(y) \,dy
	\\&=
	\sum_{j=1}^{N} |x|^{m_j} \psi_j \left( \frac{x}{|x|} \right) \left(  \sum_{i=1}^d  \int_\Gamma b_i^{(j)} \left( \frac{y}{|y|} \right) |y|^{1-d-m_j} f_i(y)\,dy \right)
\end{align*}
is a combination of homogeneous harmonic functions of degree up to $m_N$.

\noindent\textbf{Smallness of $I_2$}.
From Lemma \ref{lem-221220-1021} and \eqref{eqn-230626-1019}-\eqref{eqn-230214-0103} with $R$ replaced by $2|x|$, we have
\begin{equation*}
	\begin{split}
		| I_2 |
		&\leq
		C \sum_{j=1}^N |x|^{m_j} \left| \psi_j \right| \int_{\Gamma \cap B_{2 |x|}} \left| b_i^{(j)} \right| \left( y/|y| \right) |y|^{1-d-m_j} |f| \, dy
		\\&\leq
		C \sum_{j=1}^N |x|^{m_j} \lambda_j^{d/4} (2|x|)^{\mu - m_j  + \beta}
		\leq
		C |x|^{\mu + \beta}.
	\end{split}
\end{equation*}

\noindent\textbf{Smallness of $I_3$}. 
By \eqref{eqn-221221-0701} and \eqref{eqn-221222-1006},
\begin{align}
	|I_{3}|
	&\leq 
	C |x|^{m_{N+1}} \int_{\Gamma \cap \{2|x| \leq |y| < R\} } |y|^{1-d-m_{N+1} + \beta} |D u(y)|\,dy \nonumber
	\\&\leq
	C |x|^{m_{N+1}} \sum_{l=1}^{\log_2(R/|x|) } (2^l |x|)^{1-d-m_{N+1} + \beta} (2^l |x|)^d \fint_{\Gamma \cap (B_{2^{l+1} |x|} \setminus B_{2^l |x|}) } |D u(y)|\,dy \nonumber
	\\&\leq
	C |x|^{m_{N+1}} \sum_{l=1}^{\log_2(R/|x|) } (2^l |x|)^{1-m_{N+1} + \beta} \left( \fint_{\Gamma \cap (B_{2^{l+1} |x|} \setminus B_{2^l |x|}) } |D u(y)|^2 \,dy \right)^{1/2}. \label{eqn-230626-1046}
\end{align}
Again by the Caccioppoli inequality and $u=O(|x|^\mu)$, we obtain
\begin{align*}
	\text{RHS of \eqref{eqn-230626-1046}}
	&\leq
	C |x|^{m_{N+1}} \sum_{l=1}^{\log_2(R/|x|) } (2^l |x|)^{1-m_{N+1} + \beta} ( 2^{l}|x| )^{-1} \|u\|_{L^\infty (\Gamma \cap (B_{2^{l+1} |x|} \setminus B_{2^l |x|}))}
	\\&\leq
	C |x|^{m_{N+1}} \sum_{l=1}^{\log_2(R/|x|) } (2^l |x|)^{-m_{N+1} + \beta} (2^l |x|)^{ \mu }
	\\&\leq
	C |x|^{\mu + \beta} \sum_{l=1}^{\log_2(R/|x|) } 2^{l(\mu - m_{N+1}  + \beta)}
	\\&\leq
	C |x|^{\mu + \beta} \left( 2^{\log_2(R/|x|) ( \mu - m_{N+1} + \beta)} + 1\right) = C R^{\mu - m_{N+1}  + \beta} |x|^{m_{N+1}} + C|x|^{\mu + \beta}.
\end{align*}

\noindent\textbf{Smallness of $I_4$}. Note that the kernel $K_i(x,y)$ has two singular points for $y \in B_{2|x|}$: at $y=x$ and  $y=0$. This motivates us to split
\begin{equation*}
	I_4
	= 
	\int_{\Gamma \cap (B_{2|x|}(x) \setminus B_{|x|/10}(x))} K_i(x,y)f_i(y) \,dy + \int_{\Gamma \cap B_{|x|/10}(x)} K_i(x,y)f_i(y) \,dy
	=:
	I_{41} + I_{42}.
\end{equation*}

For $I_{41}$, noting \eqref{eqn-221227-0835}, \eqref{eqn-221222-1006}, and $|x|\approx |y| \approx |x-y|$ for all $y \in B_{2|x|}(x) \setminus B_{|x|/10}(x)$,
\begin{align*}
	| I_{41} |
	&\leq
	C \int_{\Gamma \cap (B_{2|x|}(x) \setminus B_{|x|/10}(x))} (|x-y|^{-1} + |y|^{-1}) |x-y|^{2-d} |y|^\beta |Du(y)|\,dy
	\\&\leq
	C |x|^{1-d+\beta} \int_{B_{2|x|}(x)} |Du(y)| \,dy
	\\&\leq
	C |x|^{1-d+\beta} |x|^d (\fint_{B_{2|x|}(x)} |Du|^2)^{1/2}
	\\&\leq
	C |x|^{1-d+\beta} |x|^d |x|^{-1} \|u\|_{L^\infty B_{4|x|}(x)}
	\leq
	C |x|^{\beta + \mu}.
\end{align*}
Here in the last line, we also used the Caccioppoli inequality and $u = O(|x|^\mu)$.
%
%

For $I_{42}$, by \eqref{eqn-221227-0835}, \eqref{eqn-221222-1006}, and $|y| \approx |x| \geq 10 |x-y|$,
%
\begin{equation*}
	|I_{42}| 
	\leq 
	C \int_{\Gamma \cap B_{|x|/10}(x)} \frac{1}{|x-y|^{d-1}} |y|^\beta |Du(y)| \,dy
	\leq
	C |x|^\beta \int_{\Gamma \cap B_{2|x|}} \frac{1}{|x-y|^{d-1}} |Du(y)| \,dy.
\end{equation*}
Hence, for $|x| \leq r$,
\begin{align*}
	| I_{42} |
	&\leq
	C r^\beta \left| \int_{\Gamma \cap B_{2|x|}} \frac{1}{|x-y|^{d-1}} |Du(y)| \,dy \right| \nonumber
	\\&=
	C r^\beta \left| \int_{\Gamma \cap B_{3r}(x) \cap B_{2r}(0)} \frac{1}{|x-y|^{d-1}} |Du(y)| \,dy \right| \nonumber
	\\&\leq
	C r^\beta \left| \left( \left( |z|^{-d+1} 1_{B_{3r}} \right) \ast \left( |Du|1_{B_{2r} \cap \Gamma}\right) \right) (x) \right|. 
\end{align*}
By Young's inequality for convolution, a reverse H\"older's inequality for $Du$ (cf. \cite{MR3099262}), the Caccioppoli inequality, and $u = O(|x|^\mu)$, we obtain that for some $q<d/(d-1)$ to be fixed later, a small $\epsi>0$ coming from reverse H\"older's inequality, and $p$ satisfying $1 + 1/p = 1/q + 1/(2+\epsi)$,
\begin{align*}
	\|I_{42}\|_{L^p (B_r \cap \Gamma)}
	&\leq
	C r^\beta \left\| 1_{B_{3r}}|z|^{-d+1} \right\|_{L^q} \left\| |Du|1_{B_{2r} \cap \Gamma} \right\|_{L^{2+\epsi}} 
	\\&\leq
	C r^\beta r^{-d+1+d/q} r^{d/(2+\epsi)} \left( \fint_{B_{4r}\cap \Omega} |Du|^2 \right)^{1/2}
	\\&\leq
	C r^\beta r^{-d+1+d/q} r^{d/(2+\epsi)} r^{-1} \| u \|_{L^\infty (\Gamma \cap B_{8r})}
	\leq
	C r^{-d + d/q + d/(2+\epsi)} r^{\mu + \beta}.
\end{align*}
That is,
\begin{equation*}
	\left( \fint_{\Gamma \cap B_r} |I_{42}|^p \right)^{1/p} \leq C r^{\mu + \beta}.
\end{equation*}
Finally, we can make $p > 2d/(d-2)$ by choosing $q$ to be sufficiently close to $d/(d-1)$.

\subsection{Estimating the term \texorpdfstring{$II$}{II}} 

Similar to the treatment of $I$, we split
\begin{align*}
	II 
	&=
	\int_{\p\Gamma} k^{(N)}(x,y) h(y) \,d\sigma_y - \int_{\p\Gamma \cap B_{2|x|}} k^{(N)}(x,y) h(y) \,d\sigma_y 
	\\&\quad + 
	\int_{\p\Gamma \cap B_{2|x|^c}} ( k(x,y) - k^{(N)}(x,y) ) h(y) \,d\sigma_y + \int_{\p\Gamma \cap B_{2|x|}} k(x,y) h(y) \,d\sigma_y
	\\&=:
	II_1 + II_2 + II_3 + II_4,
\end{align*}
and further
\begin{align*}
	II_4
	&=
	\int_{\p\Gamma \cap (B_{2|x|}(x) \setminus B_{|x|/10}(x))} k(x,y) h(y) \,d\sigma_y + \int_{\p\Gamma \cap B_{|x|/10}(x)} k(x,y) h(y) \,d\sigma_y
	\\&=:
	II_{41} + II_{42}.
\end{align*}
The estimates for $II_1, II_2, II_3$, and $II_{41}$ are very similar to those of the corresponding terms in $I$. Actually, the estimates here are simpler since a point-wise bound \eqref{eqn-221229-1157} for $h$ is available, instead of merely an $L^2$ bound for $Du$ from Caccioppoli's inequality. For $II_1$, formally,
\begin{align*}
	II_1
	&=
	\int_{\p\Gamma} \sum_{j\leq N} b^{(j)} \left( \frac{y}{|y|} \right) |y|^{1-d-m_j} |x|{m_j} \psi_j \left( \frac{x}{|x|} \right) h(y) \,d\sigma_y
	\\&=
	\sum_{j\leq N} |x|^{m_j} \psi_j \left( \frac{x}{|x|} \right) \int_{\p\Gamma} b^{(j)} \left( \frac{y}{|y|} \right) |y|^{1-d-m_j} h(y) \,d\sigma_y,
\end{align*}
which is a combination of homogeneous solutions on $\Gamma$ with homogeneity at most $m_N$. Such formal computation is rigorous since all integrands are in $L^1$, which we check below. By \eqref{eqn-230315-1042} and \eqref{eqn-221229-1157},
\begin{align*}
	\int_{\p\Gamma} |b^{(j)}| |y|^{1-d-m_j} |h| \,d\sigma_y
	&\leq
	C \int_{\p\Gamma \cap B_R} |y|^{1-d-m_j} |y|^{\mu + \alpha\alpha_0} \,d\sigma_y
	\\&\leq
	C \int_0^R r^{1-d-m_j} r^{\mu + \alpha\alpha_0} r^{d-2} \,dr
	= 
	CR^{\mu - m_j + \alpha\alpha_0}.
\end{align*}
As explained when estimating $I_2$, similar computation yields $|II_{2}| \leq C|x|^{\mu + \alpha\alpha_0}$.

For $II_{3}$, we use \eqref{eqn-221221-0701} and \eqref{eqn-221229-1157} to obtain
\begin{align*}
	| II_{3} | 
	&\leq 
	C |x|^{m_{N+1}} \int_{\p\Gamma \cap B_{2|x|}^c} |y|^{1-d-m_{N+1}} |y|^{\mu + \alpha\alpha_0}\,d\sigma_y
	\\&\leq
	C |x|^{m_{N+1}} \int_{2|x|}^R r^{\mu - m_{N+1} + \alpha\alpha_0 - 1} \,dr
	\leq
	C |x|^{\mu + \alpha\alpha_0} + C |x|^{m_{N+1}}.
\end{align*}
For $II_{41}$, from \eqref{eqn-221227-0835}, \eqref{eqn-221229-1157}, and the fact that $|x|\approx |y|\approx |x-y|$ for any $y \in B_{2|x|}(x) \setminus B_{|x|/10}(x) )$, we obtain
\begin{align*}
	| II_{41} |
	\leq
	C \int_{\p\Gamma \cap ( B_{2|x|}(x) \setminus B_{|x|/10}(x) )} (|y|^{-1} + |x-y|^{-1}) |x-y|^{2-d} |y|^{\mu + \alpha\alpha_0} \, d\sigma_y
	\leq
	C |x|^{\mu + \alpha\alpha_0}.	
\end{align*}

Finally, we estimate $II_{42}$, which is different from estimating $I_{42}$. Here, using \eqref{eqn-221227-0816-1} instead of \eqref{eqn-221227-0835}, then noting \eqref{eqn-221229-1157} and the fact that $|y| \approx |x| \geq 10 |x-y|$ for $y\in B_{|x|/10}(x)$, we obtain
\begin{align*}
	| II_{42}|
	\leq
	C \int_{\p\Gamma \cap B_{|x|/10}(x) } \frac{\delta(x)}{|x-y|^d} |y|^{\mu + \alpha\alpha_0}\,d\sigma_y
	\approx
	C |x|^{\mu + \alpha\alpha_0} \int_{\p\Gamma \cap B_{|x|/10}(x) } \frac{\delta(x)}{|x-y|^d} \,d\sigma_y.
\end{align*}
Let $x^* \in \p\Gamma$ be the point with $|x-x^*| = \delta(x)$. Since $y \in \p\Gamma$, by definition and the triangle inequality,
\begin{equation*}
	|x-y| \geq \delta(x) (=\dist(x,\p\Gamma)), \quad |x^* - y| \leq |x-y| + \delta(x) \leq 2 |x-y|.
\end{equation*}
Hence,
\begin{align*}
	| II_{42} |
	&\leq
	C |x|^{\mu + \alpha\alpha_0} \int_{\p\Gamma \cap B_{|x|/10}(x) } \frac{\delta(x)}{\delta(x)^d + |x^* - y|^d} \,d\sigma_y
	\\&\leq
	C |x|^{\mu + \alpha\alpha_0} \int_{\bR^{d-1}} \frac{\delta(x)}{\delta(x)^d + |y'|^d}\,dy'
	=
	C |x|^{\mu + \alpha\alpha_0} \int_{\bR^{d-1}} \frac{1}{1 + |y'|^d}\,d y' \leq C |x|^{\mu + \alpha\alpha_0}. 
\end{align*}

\subsection{Concluding the proof of Proposition \ref{prop-221228-0952}} 

With all above, we have constructed $P = \sum_{j=1}^N C_j |x|^{m_j} \psi_j\left( x/|x| \right)$, such that
\begin{equation*}
	\left( \fint_{\Gamma \cap B_r} | u - P |^p  \right)^{1/p} \leq C r^{\min\{ \mu + \alpha_0 \min \{\alpha,\beta\}, m_{N+1} \} }.
\end{equation*}

When $\mu \in (m_N, m_{N+1})$, clearly we must have $P = 0$ in order to match the vanishing order. When $\mu = m_N$, i.e., $u(x) = O(|x|^{m_N})$, similarly $P$ cannot have any term with homogeneity lower than $m_N$. Hence, $P$ is homogeneous of degree $m_N$. Choosing $\psi_N$ from the eigenspace properly, we have $P(\widetilde{x}) = C |\widetilde{x}|^{m_N} \psi_N\left( \widetilde{x}/|\widetilde{x}| \right)$. This finishes the proof of Proposition \ref{prop-221228-0952}.

\section{Uniqueness and non-uniqueness of blow-ups} \label{sec-221223-0602}

As mentioned in the introduction, Theorem \ref{thm-230225-1127} implies the existence of Almgren blow-ups (along subsequences) for any non-trivial harmonic functions vanishing locally on the boundary near a conical point. When $d=2$, such limit is also unique.
\begin{proposition} 
	Let $\Omega \subset \bR^2$ and $0\in\p\Omega$ be a conical point with a tangent cone $\Gamma$. Suppose that $u$ is a nontrivial harmonic function vanishing locally on $\p\Omega$ near $0$, with $u = O(|x|^N)$ for some $N>0$. Then there exists a homogeneous harmonic function $P$ on $\Gamma$, vanishing on $\p\Gamma$, such that for $u_r$ defined in \eqref{eqn-230627-0304}, we have $u_r \rightarrow P$ weakly in $H^1(B_1)$ and strongly in $C^{0,\alpha_0}(B_1)$, where $\alpha_0$ is a constant depending only $\Gamma$.
\end{proposition}
As before, in the statement we do not distinguish $u$ and $P$ with their zero extensions.

	We give a sketch of the proof. First, as explained in the introduction, all subsequence limits of $u_r$ have to be $L^2$-normalized on $\p B_1$, lying in an eigenspace of $\lambda_N$ determined by $\lim_{r\rightarrow 0} N_u(r)$. Now from the fact that the eigenvalue $\lambda_N$ is simple, there exists an $L^2$-normalized eigenfunction $\psi_N$, such that all possible Almgren blow-ups along subsequences have to be either $+ \psi_N$ or $-\psi_N$. Noting that the full blowup sequence $u_r$ varies continuously with respect to $r$ in $L^2(\p B_1)$, the limit has to be unique.

When the dimension is three or more, higher eigenvalues need not be simple, so the argument fails. This suggests that Almgren blowup sequences might ``rotate'' within these eigenspaces, leading to non-unique limits. In this section, we confirm that this actually can happen by constructing the example promised in Theorem \ref{thm-230315-0259}.


\subsection{Setup}

We construct $\Omega$ by intersecting cones. For $k=1,2,\ldots$, let $\alpha_k,\beta_k \in (0,\pi/2)$ be numbers satisfying
\begin{equation} \label{eqn-230629-0921}
	\beta_1 < \alpha_1 < \cdots < \beta_k < \alpha_k < \cdots \rightarrow \pi/2.
\end{equation}
The values of $\alpha_k,\beta_k$ are not important. For instance, here we can fix $\alpha_k = \pi/2 - 2^{-(2k+2)}$ and $\beta_k = \pi/2 - 2^{-(2k+1)}$. Let
\begin{equation*}
	\Gamma_k 
	= 
	\left\{(x,y,z): z > \sqrt{\left( \frac{x}{\tan\beta_k} \right)^2 + \left( \frac{y}{\tan\alpha_k} \right)^2} \right\}.
\end{equation*}
be an elliptic cone with opening angles $2\alpha_k$ and $2\beta_k$. For a sequence of numbers with
\begin{equation*} 
	\{\delta_k\}_{k=1}^\infty \subset (0,1), \quad \delta_k \downarrow 0
\end{equation*}
to be chosen later, we set
\begin{equation*} 
	\Omega 
	=
	\bigcap_{k=1}^\infty (O_k \Gamma_k - \delta_k),
\end{equation*}
where
\begin{equation*}
	O_k = I_3 = \begin{pmatrix}
		1 & 0 & 0\\0 & 1 & 0\\0 & 0 & 1
	\end{pmatrix} \,\,\text{if $k$ odd},
	\quad
	O_k = O = \begin{pmatrix}
		0 & 1 & 0\\-1 & 0 & 0\\0 & 0 & 1
	\end{pmatrix}\,\,\text{if $k$ even}.
\end{equation*}
Here and throughout this section, we use the abbreviation
\begin{equation*}
	\Gamma_k - \delta_k = \Gamma_k - \delta_k (0,0,1) = \Gamma_k - \delta_k e_z
\end{equation*}
for cones shifted along the $z$ axis. Clearly, this $\Omega$ is convex with $0\in\p\Omega$, and is symmetric with respect to reflections about the $x$ and $y$ axes.

\begin{figure}[h]	
	\begin{tikzpicture}
		\draw [->] (0,-3.1) -- (0, 3.1) node (taxis) [right] {$z$};
		\draw [->] (-8.1,0) -- (8.1,0) node (xaxis) [above] {$x$};
		\draw (0,-3) -- (-6,3); \draw (0,-3) -- (6,3); \draw [domain=90:135] plot ({cos(\x)/2}, {sin(\x)/2 - 3});
		\node[right] at (0,-3) {$-\delta_1$}; \node[above] at (-0.2,-2.6) {$\beta_1$}; \node[left] at (6,3) {$\Gamma_1$};
		
		\draw (0,-1) -- (-8,3); \draw (0,-1) -- (8,3); \draw [domain=90:150] plot ({cos(\x)/3}, {sin(\x)/3 - 1});
		\node[right] at (0,-1) {$-\delta_2$}; \node[above] at (-0.2,-0.8) {$\alpha_2$}; \node[left] at (8,3.1) {$O_2\Gamma_2$};
		
		\draw (0,-0.2) -- (-6,1.3); \draw (0,-0.2) -- (6,1.3); \draw [domain=90:165] plot ({cos(\x)/3}, {sin(\x)/3 - 0.2});
		\node[right] at (0,-0.3) {$-\delta_3$}; \node[above] at (-0.2,0) {$\beta_3$}; \node[left] at (6,1.4) {$\Gamma_3$};
		%
		\fill [gray,opacity=0.3] (-6,3) -- (-4,1) -- (-3.2,0.6) -- (0,-0.2) -- (3.2,0.6) -- (4,1) -- (6,3);
	\end{tikzpicture}
\end{figure}

Let $\lambda_{x,k}$ and $\lambda_{y,k}$ be the first Dirichlet eigenvalues of the bisected cross-sections $\Gamma_k \cap \p B_1 \cap \{x>0\}$ and $\Gamma_k \cap \p B_1 \cap \{y>0\}$, respectively. Due to the eccentricity of the elliptical region $\Gamma_k \cap \p B_1$, we have $\lambda_{x,k} < \lambda_{y,k}$. In the next section, we sketch a proof of this fact via an elementary perturbation argument from a spherical cap. A more detailed exposition (in the case of ellipses in the plane, but the approach is the same) can be found in \cite{henry_2005}.

\subsection{The spectrum of perturbations of spherical caps} \label{sec-230627-0549}

\newcommand{\R}{\mathbb{R}}

Let $E_0 \subset S^{d-1}\subset \R^d$ be a spherical cap of the form $E_0 = \{x : |x| = 1, x_d > s\}$ for a fixed $s \in (-1, 1)$. We parametrize $S^{d-1}\setminus\{(0, \ldots, 0, 1)\}$ by $(\psi, \theta) \in [0, \pi) \times S^{d-2} = T_{(1, 0, \ldots, 0)}S^{d-1}$ via the exponential map, and write $g$ for the round metric. Let $\phi_t : S^{d-1} \times [0, T) \rightarrow S^{d-1}$ be a family of diffeomorphisms smooth in both parameters and with $\phi_0(x) = x$. Set $E_t = \phi_t(E_0)$.

Now consider the Dirichlet eigenvalues $\{\lambda_k(E_t)\}_{k = 1}^\infty$ of these domains, i.e. the nondecreasing sequences of numbers for which
\[
\begin{cases}
	- \Delta_{S^{d-1}} v_k = \lambda_k(E_t) v_k & \text{ on } E_t \\
	v_k = 0 & \text{ on } \partial E_t.
\end{cases}
\]
At $t = 0$, $\lambda_k(E_0)$ have a straightforward structure which may be verified by separation of variables: $\lambda_1(E_0)$ is simple, then $\lambda_2(E_0) = \ldots = \lambda_{d}(E_0) < \lambda_{d+1}(E_0)$ with an orthonormal (in $L^2(E_0)$) basis of $(d-1)$ eigenfunctions $\{v_i\}_{i = 1}^{d-1}$.

The eigenvalues $\{\lambda_k(E_t)\}_{k = 2}^d$ of $E_t$ form, for $t$ small, a union of $C^1$ curves of the following form: let
\[
m_{ij} = - \int_{\partial E_0} g(\nabla v_i, \nabla v_j) g(V, \nu) dA,
\]
where $A$ is the surface measure on $\partial E_0$ and $v_i$ are the basis of second eigenfunctions, $V = \partial_t \phi_t|_{t = 0}$ (this is a vector field), and $\nu$ is the outward unit normal vector to $E_0$. This is an $(d-1)$-dimensional symmetric matrix, with eigenvalues $\mu_1 \leq \ldots \leq \mu_{d-1}$. Then
\[
\lambda_k(E_t) = \lambda_2(E_0) + t \mu_k + O(t^2).
\]
This formula can be found in \cite{R} in the case of subsets $\R^d$, but remains valid over any Riemannian manifold by the same argument after a direct computation of the variation of the Dirichlet and volume integrals (\cite{HP} carries out such computations). As a consequence, if the numbers $\mu_k$ are all distinct, then there is a $t_0 > 0$ such that for $t \in (0, t_0)$ the eigenvalues $\lambda_2(E_t), \ldots, \lambda_n(E_t)$ are simple.

When $d = 3$, the eigenfunctions $v_1, v_2$, by separation of variables, are easily seen to be of the form $v_1(\psi, \theta) = q(\psi)\cos(\theta)$, $v_2(\psi, \theta) = q(\psi)\sin(\theta)$  (after a rotation), for a smooth function $q$ which is positive on $[0, \arccos s)$ and vanishes at $\psi = \arccos s$. This gives the explicit formula
\[
m = - (q')^2 \int \left(\begin{matrix}
	\cos^2 \theta & \cos \theta \sin \theta \\ \cos \theta \sin \theta & \sin^2 \theta
\end{matrix} \right) g(V, \nu) dA.
\]
Then $\mu_1 = \mu_2$ if and only if this matrix is a multiple of the identity, or equivalently if
\[
\int_{\partial E_0} \sin (2\theta) g(V, \nu) dA = 0 \qquad \text{ and } \qquad \int_{\partial E_0} \cos (2\theta) g(V, \nu) dA = 0.
\]
If $\phi_t(\phi, \theta, t) = (\phi(1 + h(\theta, t)), \theta)$ where $h(\theta, 0) = 0$, $h$ is even and $\pi$-periodic in $\theta$, and has $\p_t h(\cdot, 0)$ strictly decreasing on $[0, \pi/2]$, then the second integral is positive and $E_t$ has $\lambda_2(E_t) < \lambda_3(E_t)$ for $t \in (0, t_0)$. Moreover, the domains $E_t$ are symmetric across the planes $\theta = 0$ and $\theta = \pi/2$, so odd reflections of the first eigenfunctions of the bisected domains $E_t \cap \{ \theta \in (0, \pi) \}, E_t \cap \{\theta \in (-\pi/2, \pi/2) \}$ give eigenfunctions on $E_t$. As all the eigenvalues are continuous in $t$, these must be the second and third eigenfunctions: in particular, $\lambda_1(E_t \cap \{ \theta \in (0, \pi) \}) \neq \lambda_1(E_t \cap \{\theta \in (-\pi/2, \pi/2) \})$. A more careful examination of the matrix $m$ shows that in fact $\lambda_1(E_t \cap \{ \theta \in (0, \pi) \})$ is the larger of the two.

It is then easy to see that if given $E_0$ a hemisphere, one may construct a $\phi_t$ of this form for which each $E_t$ is the cross-section of an elliptic cone with opening angles $\alpha = \pi/2 - t$ and $\beta = \pi/2 - 2t$ (along the $\theta = 0$ and $\theta = \pi/2$ axes, respectively). We conclude that $\lambda_{x,k} < \lambda_{y,k}$ for the cones $\Gamma_k$ above so long as $\beta_k - \alpha_k$ is small enough.

\subsection{Back to Example}

The discussion in Section \ref{sec-230627-0549} shows that the second eigenvalue of the hemisphere $\p B_1 \cap \{(x,y,z): z>0\}$, which equals $6$ with a multiplicity of $2$, splits into two simple second and third eigenvalues on every $\Gamma_k \cap \p B_1$.
Furthermore, from $\lambda_{x,k} < \lambda_{y,k}$, we have the same order for their characteristic constants (see \eqref{eqn-221018-0432}), i.e., $m_{x,k} < m_{y,k}$. Recall that by separation of variables, $m_{x,k}$ and $m_{y,k}$ equal to the homogeneity of the corresponding extended harmonic functions on $\Gamma_k$.

Now find harmonic functions $u_1, u_2$ solving the Dirichlet problems
\begin{equation*}
	\begin{cases}
		\Delta u_1 = 0 \,\, \Omega \cap B_1 \cap \{x>0\},\\
		u_1 = 0 \,\, ( \p\Omega \cup \{x=0\} )\cap B_1,\\
		u_1 = 1 \,\, \p B_1 \cap (\Omega \cap \{x>0\})
	\end{cases}
	\,\,\text{and}\,\,
	\begin{cases}
		\Delta u_2 = 0 \,\, \Omega \cap B_1 \cap \{y>0\},\\
		u_2 = 0 \,\, ( \p\Omega \cup \{y=0\} )\cap B_1,\\
		u_2 = 1 \,\, \p B_1 \cap (\Omega \cap \{y>0\}).
	\end{cases}
\end{equation*}
These are in $W^{1,2}(B_t \cap \Omega)$ for $t < 1$. Take odd extensions of $u_1, u_2$ with respect to $x=0$ and $y=0$, respectively, and then extend both by zero outside $\Omega$. Still denote the resulting functions by $u_1, u_2$. It is not difficult to see that
\begin{equation*}
	N_{u_1}(r), N_{u_2}(r) \downarrow 4^{2\times 2} = 256, \quad \text{as}\,\,r\downarrow 0
\end{equation*}
and
\begin{equation*}
	\frac{ u_1( r \cdot ) }{ (\fint_{\p B_{r}} |u_1|^2)^{1/2} } 
	\rightarrow
	4\sqrt{2/\pi} xz,
	\quad
	\frac{ u_2( r \cdot ) }{ (\fint_{\p B_{r}} |u_2|^2)^{1/2} } 
	\rightarrow
	4\sqrt{2/\pi} yz,
	\quad
	\text{as}\,\,r\downarrow 0,
\end{equation*}
which are $L^2$-normalized eigenfunctions associated with the second eigenvalue $\lambda_2 = 6$ on the hemisphere $\p B_1 \cap \{(x,y,z): z>0\}$. In the following, we show that the desired ``rotation'' occurs for $u = u_1 + u_2$ by choosing $\delta_k$ properly. We consider auxiliary domains $\Omega_k, \widehat{\Omega}_{k}$ and auxiliary functions $u^{(k)}_1, u^{(k)}_2$.
Here,
\begin{equation*}
	\Omega_k := \cap_{j\leq k} (O_j\Gamma_j - \delta_j) \cap (O_{k+1}\Gamma_{k+1})
	\quad \text{and} \quad
	\widehat{\Omega}_{k} := \cap_{j \leq k} (O_j \Gamma_j - \delta_j).
\end{equation*}
It is not difficult to see that
\begin{equation} \label{eqn-230629-0821}
	\Omega_1 \subset \Omega_2 \subset \cdots \subset \Omega_k \subset \cdots\subset\Omega \subset \cdots\subset \widehat{\Omega}_k \subset \cdots \subset \widehat{\Omega}_2 \subset \widehat{\Omega}_1.
\end{equation}
The functions $u^{(k)}_1, u^{(k)}_2$ are solutions to
\begin{equation} \label{eqn-230629-1233}
	\begin{cases}
		\Delta u^{(k)}_1 = 0 \,\, \Omega_{k-1} \cap B_1 \cap \{x>0\},\\
		u^{(k)}_1 = 0 \,\, ( \p \Omega_{k-1} \cup \{x=0\} )\cap B_1,\\
		u^{(k)}_1 = 1 \,\, \p B_1 \cap (\Omega_{k-1} \cap \{x>0\}),
	\end{cases}
	\text{and}\,\,
	\begin{cases}
		\Delta u^{(k)}_2 = 0 \,\, \widehat{\Omega}_{k} \cap B_1 \cap \{y>0\},\\
		u^{(k)}_2 = 0 \,\, ( \p \widehat{\Omega}_{k} \cup \{y=0\} )\cap B_1,\\
		u^{(k)}_2 = 1 \,\, \p B_1 \cap (\widehat{\Omega}_{k} \cap \{y>0\}),
	\end{cases} 
	\text{if} \,\, k \,\,\text{odd},
\end{equation}
or,
\begin{equation} \label{eqn-230629-0859}
	\begin{cases}
		\Delta u^{(k)}_1 = 0 \,\, \widehat{\Omega}_{k} \cap B_1 \cap \{x>0\},\\
		u^{(k)}_1 = 0 \,\, ( \p \widehat{\Omega}_{k} \cup \{x=0\} )\cap B_1,\\
		u^{(k)}_1 = 1 \,\, \p B_1 \cap (\widehat{\Omega}_{k} \cap \{x>0\}),
	\end{cases} 
	\text{and}\,\,
	\begin{cases}
		\Delta u^{(k)}_2 = 0 \,\, \Omega_{k-1} \cap B_1 \cap \{y>0\},\\
		u^{(k)}_2 = 0 \,\, ( \p \Omega_{k-1} \cup \{y=0\} )\cap B_1,\\
		u^{(k)}_2 = 1 \,\, \p B_1 \cap (\Omega_{k-1} \cap \{y>0\}),
	\end{cases} \text{if} \,\, k \,\,\text{even}.
\end{equation}
Like $u_1, u_2$, in the following we first take odd extensions of $u^{(k)}_1, u^{(k)}_2$ with respect to $x=0$ and $y=0$, respectively, and then extend both by zero outside. Still denote the resulting functions by $u^{(k)}_1, u^{(k)}_2$. 

Below, we attempt to choose a sequence of $\delta_k$ decreasing to zero, such that, due to the alternation of $m_{x,k}$ and $m_{y,k}$,
	\begin{equation} \label{eqn-230628-1205}
	\frac{ \fint_{\p B_{\delta_k}} |u^{(k)}_1|^2 }{ \fint_{\p B_{\delta_k}} |u^{(k)}_2|^2 } > k
	\,\,\text{if}\,\,k\,\,\text{odd and}\,\,
	\frac{ \fint_{\p B_{\delta_k}} |u^{(k)}_1|^2 }{ \fint_{\p B_{\delta_k}} |u^{(k)}_2|^2 } < \frac{1}{k}		\,\,\text{if}\,\,k\,\,\text{even}.
\end{equation}
In Section \ref{sec-230627-1150} we show that \eqref{eqn-230628-1205} together with a comparison, which comes from \eqref{eqn-230629-0821}, imply that the desired ``rotation'' occurs.

\subsection{Choosing \texorpdfstring{$\delta_k$}{the radii}}

\hfill\\ 
\textbf{Starting point $k=1$:} 
Recall that $u^{(1)}_1, u^{(1)}_2$ solve \eqref{eqn-230629-1233} with $\Omega_0 = \Gamma_1$ and $\widehat{\Omega}_1 = \Gamma_1 - \delta_1$. Denote $\psi_{x,1}$ to be the (odd extension with respect to $x=0$ of) $L^2$-normalized leading Dirichlet eigenfunction on $\p B_1 \cap \Gamma_1 \cap \{x>0\}$. From the orthogonality properties of the spherical harmonics, it is easy to see that
\begin{equation} \label{eqn-230131-0656-1}
	\left( \fint_{\p B_r \cap \Gamma_1} |u^{(1)}_1|^2 \right)^{1/2} 
	\geq
	r^{m_{x,1}} \left( \fint_{\p B_1 \cap \Gamma_1} u^{(1)}_1\psi_{x,1} \right)^{1/2} 
	\geq
	C^{-1} r^{m_{x,1}} , \quad \forall r<1.
\end{equation}
For $u^{(1)}_2$, by Almgren's monotonicity formula on the convex domain (actually, cone) $\widehat{\Omega}_1$, centered at $-\delta_1 e_z \in \p \widehat{\Omega}_1$, and the fact that the lowest non-orthogonal mode of $u^{(1)}_2$ is $\lambda_{y,1}$, we have
\begin{equation} \label{eqn-230629-1240}
	\frac{\fint_{\p B_{r}( - \delta_1 e_z)} |u_2^{(1)}|^2}{\fint_{\p B_{r/4}( - \delta_1 e_z)} |u_2^{(1)}|^2} = N_{u^{(1)}_2}(r) \geq N_{u^{(1)}_2}(0) = 4^{2 m_{y,1}}, \quad \forall r\in (0,1-\delta_1).
\end{equation}
Here and also later in this section, we abuse the notation $N_{u^{(1)}_2}(r)$ which includes a shift of the center.
Iterating \eqref{eqn-230629-1240}, we obtain, for all $r < 3/4$ and $\delta_1 < 1/4$ (away from $1$ is enough),
\begin{equation*}
	\left( \fint_{\p B_r(-\delta_1 e_z)} |u^{(1)}_2|^2 \right)^{1/2}
	\leq 
	C r^{m_{y,1}} \left( \fint_{\p B_{3/4} (-\delta_1 e_z)} |u^{(1)}_2|^2 \right)^{1/2}
	\leq
	C r^{m_{y,1}} .
\end{equation*}
By the subharmonicity of $|u_2^{(1)}|^2$ and the mean value property, for $r \in (\delta_1/2, 3/16)$ and $\delta_1 < 1/4$,
\begin{equation} \label{eqn-230131-0656-2}
	\begin{split}
		\left( \fint_{\p B_r} |u^{(1)}_2|^2 \right)^{1/2} 
		&\leq
		C \left( \fint_{B_{2r}} |u^{(1)}_2|^2 \right)^{1/2} 
		\\&\leq 
		C \left( \fint_{B_{4r}(-\delta_1 e_z) } |u^{(1)}_2|^2 \right)^{1/2} 
		\leq
		C \left( \fint_{\p B_{4r}(-\delta_1 e_z) } |u^{(1)}_2|^2 \right)^{1/2} 
		\leq
		C r^{m_{y,1}}.
	\end{split}
\end{equation}
Here $C$ is a constant independent of $r$ and $\delta_1$.

Combining \eqref{eqn-230131-0656-1} - \eqref{eqn-230131-0656-2} with $r = \delta_1$,
\begin{equation*}
	\frac{ \fint_{\p B_{\delta_1}} |u^{(1)}_1|^2 }{ \fint_{\p B_{\delta_1}} |u^{(1)}_2|^2 }
	\geq 
	C^{-1} \delta_1^{2(m_{x,1} - m_{y,1})}, \quad \forall \delta_1 < 1/4.
\end{equation*}
Now, noting $m_{x,1} < m_{y,1}$, we are able to choose $\delta_1$ small enough, such that $C^{-1} \delta_1^{2(m_{x,1} - m_{y,1})} > 1$.

\noindent\textbf{Given $\{\delta_j\}_{j\leq k-1}$, choose $\delta_{k}$.} By switching the roles of $x$ and $y$, without loss of generality, we can always assume that $k$ is an even number.

Recall that $u^{(k)}_1, u^{(k)}_2$ solve \eqref{eqn-230629-0859}.
The estimate of $u^{(k)}_1$ is similar to that of $u^{(1)}_2$, noting that $\widehat{\Omega}_k$ is still convex. By Almgren's monotonicity formula centered at $-\delta_{k} e_z \in \p\widehat{\Omega}_k$,
\begin{equation*}
	\frac{\fint_{\p B_{r}( - \delta_k e_z)} |u^{(k)}_1|^2}{\fint_{\p B_{r/4}( - \delta_{k} e_z)} |u^{(k)}_1|^2} 
	= 
	N_{u^{(k)}_1} (r) 
	\geq 
	N_{u^{(k)}_1}(0) 
	=
	4^{2 m_{y,k}}, \quad \forall r < 1 - \delta_k.
\end{equation*}
Here, in the last step, we used the fact that the lowest non-orthogonal mode of $u^{(k)}_1$ is $\lambda_{y,k}$, noting that the cone $O_{k}\Gamma_{k}$ was rotated by $90$ degree in $(x,y)$. Hence, following the proof of \eqref{eqn-230131-0656-2}, we obtain, whenever $\delta_k < \delta_1 < 1/4$,
\begin{equation} \label{eqn-230201-1230-1}
	\left( \fint_{\p B_r} |u^{(k)}_1|^2 \right)^{1/2}
	\leq
	C r^{m_{y,k}}, \quad \forall r\in (\delta_{k}/2 , 3/16 ).
\end{equation}
Compared to $u^{(1)}_1$, the estimate of $u^{(k)}_2$ requires some extra work since now $\Omega_{k-1}$ is not exactly a cone.
Define
\begin{equation*}
	R_k := \sup \{ r: B_r \cap \p\Omega_{k-1} \subset \p (O_{k} \Gamma_{k}) \}.
\end{equation*}
From the monotonicity of the cones coming from \eqref{eqn-230629-0921}, we have $R_k>0$. Clearly, $R_k$ depends on $\{\delta_j\}$ up to $j \leq k-1$. 
As in the proof of \eqref{eqn-230131-0656-1}, denote $\psi_{x,k}$ to be the (odd extension of) $L^2$-normalized leading Dirichlet eigenfunction on $\p B_1 \cap \Gamma_{k} \cap \{x>0\}$. Note that $O_k$ is a $90$ degree rotation in $(x,y)$ and the symmetry, the projection of $u^{(k)}_2$ onto $\psi_{x,k}$ is non-trivial, from which
\begin{equation} \label{eqn-230201-1230-2}
	\begin{split}
		\left( \fint_{\p B_r} |u^{(k)}_2|^2 \right)^{1/2}
		&=
		\left( \fint_{\p B_r \cap O_k\Gamma_k} |u^{(k)}_2|^2 \right)^{1/2}
		\\&\geq
		(\frac{r}{R_k})^{m_{x,k}} \left( \fint_{\p B_{R_k} \cap O_k\Gamma_k} u^{(k)}_2 \psi_{x,k} \right)^{1/2} 
		\geq
		C (\frac{r}{R_l})^{m_{x,k}}, \quad \forall r<R_l,
	\end{split}
\end{equation}
where $C$ is a constant depending on $\{\delta_j\}_{j\leq k-1}$. Here we used the fact that
\begin{equation*}
	\p B_r \cap \p\Omega_{l-1} \subset \p O_{l} \Gamma_{l}, \quad \forall r<R_l
\end{equation*}
coming from our definition of $R_l$.

Combining \eqref{eqn-230201-1230-1}-\eqref{eqn-230201-1230-2} and the fact that $m_{x,l+1} < m_{y,l+1}$, we can choose $\delta_k$ small enough, such that $\delta_k < R_k$ and
\begin{equation*}
	\frac{ \fint_{B_{\delta_k}} |u^{(k)}_1|^2 }{ \fint_{\p B_{\delta_k}} |u^{(k)}_2|^2 }
	\leq
	C R_k^{ - 2m_{x,k}} \delta_{k}^{2(m_{y,k} - m_{x,k})} 
	< \frac{1}{k}.
\end{equation*}
With all above, we have finished our choice of $\{\delta_k\}_{k=1}^\infty$.
%

\subsection{Conclusion of the proof of Theorem \ref{thm-230315-0259}} \label{sec-230627-1150}

Since $\Omega_{2k} \subset \Omega \subset \widehat{\Omega}_{2k+1}$, by comparison principle we have $|u_1| \geq |u^{(2k+1)}_1|$ and $|u_2| \leq |u^{(2k+1)}_2|$. Hence,
\begin{equation*}
	\limsup_{r\rightarrow \infty} \frac{ \fint_{\p B_r} |u_1|^2 }{\fint_{\p B_r} |u_2|^2}
	\geq
	\lim_{k\rightarrow \infty} \frac{ \fint_{\p B_{\delta_{2k+1}}} |u^{(2k+1)}_1|^2 }{\fint_{\p B_{\delta_{2k+1}}} |u^{(2k+1)}_2|^2} 
	\geq
	\lim_{k\rightarrow \infty} k
	= \infty.
\end{equation*}
Similarly,
\begin{equation*}
	\liminf_{r\rightarrow \infty} \frac{ \fint_{\p B_r} |u_1|^2 }{\fint_{\p B_r} |u_2|^2} 
	\leq
	\lim_{k\rightarrow \infty} \frac{ \fint_{\p B_{\delta_{2k}}} |u^{(2k)}_1|^2 }{\fint_{\p B_{\delta_{2k}}} |u^{(2k)}_2|^2} 
	\leq
	\lim_{k\rightarrow \infty} \frac{1}{k}
	= 0.
\end{equation*}
Then in strong $L^2(B_1)$ topology, $u=u_1 + u_2$ satisfies
\begin{equation*}
	\lim_{k\rightarrow \infty} \frac{ u( \delta_{2k+1} \cdot ) }{ (\fint_{\p B_{\delta_{2k+1}}} |u|^2)^{1/2} } 
	= 
	\lim_{k \rightarrow \infty} \frac{ u_1( \delta_{2k+1} \cdot ) }{ (\fint_{\p B_{\delta_{2k+1}}} |u_1|^2)^{1/2} } 
	=
	4\sqrt{2/\pi} xz
\end{equation*}
and
\begin{equation*}
	\lim_{k\rightarrow \infty} \frac{ u(\delta_{2k}\cdot) }{ (\fint_{\p B_{\delta_{2k}}} |u|^2)^{1/2} } 
	=
	\lim_{k\rightarrow \infty} \frac{ u_2(\delta_{2k}\cdot) }{ (\fint_{\p B_{\delta_{2k}}} |u_2|^2)^{1/2} } 
	=
	4\sqrt{2/\pi} yz,
\end{equation*}
which are second eigenfunctions on a hemisphere, symmetric with respect to $x$ and $y$ axis, respectively. Hence, $u$ has non-unique blowup limits near the origin. This finishes the proof of Theorem \ref{thm-230315-0259}.

%
%

\appendix
\section{Proof of gradient estimates in Lemma \ref{lem-221227-0652}}\label{sec-230315-0811}

The proof of \eqref{eqn-221227-0835} is standard. Take $R:= \frac{1}{2} \min \{|y| , |x-y|\}$. Since $0 \not \in B_{R}(y)$ and $x \notin B_{R}(y)$, we apply a local Lipschitz estimate for harmonic functions in $B_{R}(y) \cap \Gamma$ and the point-wise bound \eqref{eqn-221218-1113} to obtain
\begin{equation*}
	|\nabla_y G(x,y)| \leq C R^{-1} \sup_{z\in B_R (y) \cap \Gamma} |G(x,z)| \leq C R^{-1} |x-y|^{2-d},
\end{equation*}
which proves \eqref{eqn-221227-0835}. Here, Assumption \ref{ass-230119-0339} implies the same smoothness of $\p\Gamma \cap B_R$, and hence a local Lipschitz estimate for harmonic functions with zero Dirichlet boundary conditions. When $g$ in Assumption \ref{ass-230119-0339} is $C^{1,Dini}$, such estimate is standard. When it is semiconvex, a local Lipschitz estimate can be obtained by constructing a simple barrier coming with the exterior ball.

The proof of \eqref{eqn-221227-0816-1} requires some more work. First, by \eqref{eqn-221227-0835} and the fact that $G(x,y) = 0$ for $y \in\p\Gamma$, we immediate obtain $G(x,y) \leq C \delta(y) (|y|^{-1} + |x-y|^{-1}) |x-y|^{2-d}$. By symmetry, we also have
\begin{equation} \label{eqn-230701-0716}
	G(x,y) \leq C \delta(x) (|x|^{-1} + |x-y|^{-1}) |x-y|^{2-d}.
\end{equation}
Now, repeating the argument in the proof of \eqref{eqn-221227-0835} but using the new point-wise bound \eqref{eqn-230701-0716} instead of \eqref{eqn-221218-1113}, we reach \eqref{eqn-221227-0816-1}. The lemma is proved.

\section{Doubling indices and Almgren's frequencies on cones} \label{sec-230225-1124}
In this section, we prove Lemma \ref{lem-230225-1124}. Define the (generalized) Almgren's frequency functions
\begin{equation*}
	F(r) := \frac{r \int_{B_r} |\nabla u|^2 }{\int_{\p B_r} |u|^2}, \quad \widetilde{F}(r) := \frac{\int_{B_r} |\nabla u|^2(r^2 - |x|^2) }{\int_{B_r} |u|^2}.
\end{equation*}
By standard computation, we have
\begin{equation} \label{eqn-230227-0744-1}
	F'(r) = \frac{2r}{h^2} \left( \left(\int_{\p B_r} |u|^2\right) \left(\int_{\p B_r \cap \Gamma} (\nu\cdot \nabla u)^2\right) - \left(\int_{\p B_r} u \left( \nu \cdot \nabla \right) u\right)^2 \right)
\end{equation}
and
\begin{equation*} 
	\widetilde{F}'(r) = \frac{4}{r\widetilde{h}^2} \left( \left(\int_{B_r} |u|^2\right) \left(\int_{B_r\cap\Gamma} |x\cdot\nabla u|^2\right) - \left(\int_{B_r \cap \Gamma} u \left(x \cdot \nabla \right) u\right)^2 \right).
\end{equation*}

See for instance, \cite{MR1363203}. Note that here all contributions from $\p\Gamma$ vanish since $\Gamma$ is a cone. Moreover, by standard elliptic regularity theory we have $Du|_{\p\Omega} \in L^2$, which is understood in the sense of non-tangential limit. This guarantees that all the integration by parts in the process are justified. Now we give the proof of Lemma \ref{lem-230225-1124}.
\begin{proof}[Proof of Lemma \ref{lem-230225-1124}]
		Noting
	\begin{equation*}
		F = \frac{rD(r)}{h(r)} = \frac{r}{2} \frac{d}{dr} \log h - \frac{d-1}{2},
		\quad
		\widetilde{F} = \frac{\widetilde{D}(r)}{\widetilde{h}(r)} = r \frac{d}{dr} \log \widetilde{h} - d,
	\end{equation*}
	we have
	\begin{align*}
		N_u(4^t) 
		&= 
		\frac{ \fint_{\p B_{4^t}} |u|^2 }{ \fint_{\p B_{4^{t-1}}} |u|^2 } 
		= 
		4^{d-1} \frac{ h(4^t) }{ h(4^{t-1}) }
		=
		4^{d-1} e^{\int_{t-1}^t \frac{d}{ds}\log h(4^s) \,ds }
		\\&=
		4^{d-1} e^{\log(4) \int_{t-1}^t (r \log(h)')|_{r=4^s} \,ds }
		=
		16^{d-1} 16^{\int_{t-1}^t F(4^s) \,ds}
	\end{align*}
	and
	\begin{align*}
		\widetilde{N}_u(2^t)
		&=
		\frac{ \fint_{B_{2^t}} |u|^2 }{ \fint_{B_{2^{t-1}}} |u|^2 }
		=
		2^d \frac{ \widetilde{h} (2^t) }{ \widetilde{h} (2^{t-1}) }
		=
		2^d e^{\int_{t-1}^t \frac{d}{ds} \log \widetilde{h}(2^s) \,ds}
		\\&=
		2^d e^{\log(2) \int_{t-1}^t (r \log (\widetilde{h})')|_{r=2^s} \,ds}
		=
		4^d 2^{\int_{t-1}^t \widetilde{F}(2^s) \, ds}.
	\end{align*}
	From these, in the following we only prove the monotonicity and rigidity of $F$ and $\widetilde{F}$ since those for $N,\widetilde{N}$ naturally follow.	
	Moreover, we only prove for $F$ as the proof for $\widetilde{F}$ is almost identical. First, from \eqref{eqn-230227-0744-1} and H\"older's inequality, clearly $F' \geq 0$. Now, if $F(t)=F(s)$ for some $t>s$, by the condition for achieving ``$=$'' in H\"older's inequality,  $\frac{\p u}{\p r} = C_r u$ for all $r\in (s, t)$. Expanding as spherical harmonics, clearly this can only be true when $u$ is homogeneous in $r$. Once $u$ is homogeneous in $r$, we immediately have $F\equiv $ constant and $u$ is a homogeneous harmonic function.
\end{proof}

\bibliographystyle{plain}

\bibliography{biblio}

%

\end{document}